\documentclass[12pt,reqno]{amsart}

\usepackage{epsfig}
\usepackage{color}
\usepackage{amsmath, amsthm, amsfonts, amssymb, mathrsfs,mathtools,faktor}
\usepackage{graphicx}
\usepackage{enumerate}

\numberwithin{equation}{section}

\newcommand{\K}{{\mathbb K}}

\newtheorem{theo}{{\sc \bf Theorem}}[section]

\newtheorem{lem}[theo]{{\sc \bf Lemma}}
\newtheorem{prop}[theo]{{\sc \bf Proposition}}

\theoremstyle{definition}
\newtheorem{defin}{Definition}[section]

\begin{document}

\title{Unbounded Derivations in Bunce-Deddens-Toeplitz Algebras}

\author[Klimek]{Slawomir Klimek}
\address{Department of Mathematical Sciences,
Indiana University-Purdue University Indianapolis,
402 N. Blackford St., Indianapolis, IN 46202, U.S.A.}
\email{sklimek@math.iupui.edu}

\author[McBride]{Matt McBride}
\address{Department of Mathematics and Statistics,
Mississippi State University,
175 President's Cir., Mississippi State, MS 39762, U.S.A.}
\email{mmcbride@math.msstate.edu}

\author[Rathnayake]{Sumedha Rathnayake}
\address{Department of Mathematics,
University of Michigan,
530 Church St., Ann arbor, MI 48109, U.S.A.}
\email{sumedhar@umich.edu}

\author[Sakai]{Kaoru Sakai}
\address{Department of Mathematical Sciences,
ndiana University-Purdue University Indianapolis,
402 N. Blackford St., Indianapolis, IN 46202, U.S.A.}
\email{ksakai@iupui.edu }

\author[Wang]{Honglin Wang}
\address{Department of Mathematical Sciences,
Indiana University-Purdue University Indianapolis,
402 N. Blackford St., Indianapolis, IN 46202, U.S.A.}
\email{wanghong@imail.iu.edu}

\date{\today}

\begin{abstract}
In this paper we study decompositions and classification problems for unbounded derivations in Bunce-Deddens-Toeplitz and Bunce-Deddens algebras. We also look at implementations of these derivations on associated  GNS Hilbert spaces.  
\end{abstract}

\maketitle
\section{Introduction}

The study of derivations on C$^*$-algebras, which was started in 1953 by Kaplansky, had undergone several stages during its course:  theory of bounded derivations,  unbounded derivations and noncommutative vector-fields, according to \cite{B}. Originally motivated by research on dynamics in statistical mechanics, development of the theory of unbounded derivations in C$^*$-algebras began much later than its bounded counterpart; see \cite{Sa}. The focus was on closability, generator properties and classification of closed derivations. More recently, classification and generator properties of derivations which are well behaved with respect to the action of a locally compact group were some of the major concerns \cite{BEJ}. Additionally, derivations feature in the theory of noncommutative vector fields \cite{J}, which was inspired by Connes work on noncommutative geometry \cite{C}. 

In this paper we study classification and decompositions of unbounded derivations in Bunce-Deddens-Toeplitz and Bunce-Deddens algebras \cite{BD1}, \cite{BD2}. Given an increasing sequence $\{l_k\}_{k=0}^{\infty}$ of nonnegative integers such that $l_k$ divides $l_{k+1}$ for $k\geq 0$,  the Bunce-Deddens-Toeplitz algebra is defined as the C$^*$-algebra of operators on $\ell^2(\mathbb Z_{\geq 0})$ generated by all $l_k$-periodic weighted shifts for all $k \geq 0$. Different sequences $\{l_k\}$ may lead to the same algebras, with the classifying invariant being the supernatural number $N=\prod_{p-\textnormal{prime}} p^{\epsilon_p}$, where $\epsilon_p:=\sup\{j: \exists k\ p^j|l_k\}$. In this paper we adopt a slightly different definition of the Bunce-Deddens-Toeplitz algebra $A(N)$ associated with the supernatural number $N$ that uses $N$ more directly. We consider both finite and infinite $N$.

The algebra $\mathcal K$ of compact operators on $\ell^2(\mathbb Z_{\geq 0})$ is contained in $A(N)$ and the quotient $A(N)/ \mathcal K := B(N)$ is known as the Bunce-Deddens algebra. The structure of all those algebras is quite different depending on whether $N$ is finite or infinite. The main objects of study in this paper are densely defined derivations $d: \mathcal A(N) \rightarrow A(N)$ in the Bunce-Deddens-Toeplitz algebras, where $\mathcal A(N)$ is the subalgebra of polynomials of  $l_k$-periodic weighted shifts, as well as derivations $\delta: \mathcal B(N) \rightarrow B(N)$ in the Bunce-Deddens algebras, where $\mathcal B(N)$ is the image of $\mathcal A(N)$ under the quotient map $A(N)\to A(N)/ \mathcal K = B(N)$. 

Intriguingly, if $d: \mathcal A(N) \rightarrow A(N)$ is any derivation then $d$ preserves the ideal of compact operators $\mathcal K$, and consequently $[d]: \mathcal B(N) \rightarrow B(N)$ defined by $[d](a+ \mathcal K)= d(a)+ \mathcal K$ is a derivation in $B(N)$. It is a non-trivial problem to describe properties of the map $d\mapsto [d]$. In general, on any C$^*$-algebra, bounded derivations preserve closed ideals and so define derivations on quotients. It was proven in \cite{P} that for bounded derivations and separable C$^*$-algebras the above map is onto, i.e., derivations can be lifted from quotients in separable cases but not in general. We prove here that lifting unbounded derivations from Bunce-Deddens to Bunce-Deddens-Toeplitz algebras is always possible when $N$ is finite and conjecture that it is true for any supernatural number $N$.

The main results of this paper are that any derivation in Bunce-Deddens or Bunce-Deddens-Toeplitz algebras can be uniquely decomposed into a sum of a certain special derivation and an approximately inner derivation. The special derivations are not approximately inner, are explicitly described, and depend on whether $N$ is finite or infinite.

The algebra $A(N)$ has a natural $S^1$ action  given by scalar multiplication of the generators, see formula (\ref{rho_action}), which also quotients to $B(N)$. The key technique, like in \cite{BEJ}, is to use Fourier series decomposition with respect to this action. The Fourier components of a derivation $d$ satisfy a covariance property with respect to the $S^1$ action. It turns out that such $n$-covariant derivations can be completely classified and their properties explicitly analyzed. We then use Ces\`aro convergence of Fourier series to infer properties of $d$.

Additionally, we describe implementations of derivations in various GNS Hilbert spaces associated with the algebras. Some of those implementations can be used to construct spectral triples on  Bunce-Deddens-Toeplitz and Bunce-Deddens algebras, similarly to what was done in \cite{KMR1},\cite{KMR2}.

\section{Definitions, Notations and preliminary results.}
In this section we introduce notation and terminology used in the paper. 

\subsection{$\mathbb Z/N\mathbb Z$ rings}
A {\it supernatural number} $N$ is defined as the formal product: 
\[N= \prod_{p-\textnormal{prime}} p^{\epsilon_p}, \;\;\; \epsilon_p \in\{0,1, \cdots, \infty\}.\]
If $\sum \epsilon_p < \infty$ then $N$ is said to be a finite supernatural number (a regular natural number), otherwise it is said to be infinite.  If $N'= \prod_{p-\textnormal{prime}} p^{\epsilon_p'}$ is another supernatural number, then their product is given by:
\[NN'= \prod_{p-\textnormal{prime}} p^{\epsilon_p + \epsilon_p'}.\]
A supernatural number $N$ is said to divide $M$ if $M=NN'$ for some supernatural number $N'$, or equivalently, if $\epsilon_p(N) \leq \epsilon_p(M)$ for every prime $p$.

For the remainder of the paper we work with a fixed $N$. We let 
\[\mathcal J_N=\{j: \; j|N, j<\infty\}\]
be the set of finite divisors of $N$.
Notice that  $(\mathcal J_N, \leq)$ is a directed set where $j_1 \leq j_2$ if and only if $j_1 | j_2 |N$.

Consider the collection of  rings $\left\{\mathbb Z/ j\mathbb Z\right\}_{j \in \mathcal J_N}$ and the family of ring homomorphisms
\[\begin{aligned}\pi_{ij}: \mathbb Z/ j\mathbb Z&\rightarrow \mathbb Z/ i\mathbb Z, \;\;\;\; j\geq i\\
\pi_{ij}(x)&=x\ (\textrm{mod } i)\end{aligned}\]
satisfying 
\[\pi_{ik} = \pi_{ij} \circ \pi_{jk} \textnormal{ for all } i \leq j \leq k.\]
Then the inverse limit of the system can be denoted as:
\[\mathbb Z/N\mathbb Z:=\lim_{\underset{j\in \mathcal J_N}{\longleftarrow}} \mathbb Z/ j\mathbb Z
=\left\{\{x_j\}\in\prod\limits_{j\in \mathcal J_N}\mathbb Z/ j\mathbb Z : \pi_{ij}(x_j)=x_i\right\},\]
and let $\pi_j: \mathbb Z/N\mathbb Z \ni \{x_j\}\mapsto x_j \in \mathbb Z/ j\mathbb Z$ be the corresponding homomorphisms.
In particular, if $N$ is finite the above definition coincides with the usual meaning of the symbol $\mathbb Z/N\mathbb Z$, while if $N=p^\infty$ for a prime p, then the above limit is equal to $\mathbb Z_p$, the ring of $p$-adic integers, dee for example \cite{R}.
In general we have the following simple consequence of the Chinese Reminder Theorem.

\begin{prop}
If $N= \prod\limits_{\substack{p-\textnormal{prime} \\ {\epsilon_p \neq 0}}} p^{\epsilon_p}$, then $ \mathbb Z/N\mathbb Z \cong \prod\limits_{\substack{p-\textnormal{prime} \\ {\epsilon_p \neq 0}}}  \mathbb Z/ {p^{\epsilon_p}}\mathbb Z$. 
 \end{prop}

When the ring $ \mathbb Z/N\mathbb Z $ is equipped with the Tychonoff topology it forms a compact, abelian topological group. Thus it has a unique normalized Haar measure $d_Hx$. Also, if $N$ is an infinite supernatural number then $ \mathbb Z/N\mathbb Z$ is a Cantor set \cite{W}. 

Let $q_j: \mathbb Z \rightarrow \mathbb Z/j\mathbb Z$ be the quotient maps and let $q: \mathbb Z \rightarrow \mathbb Z/N\mathbb Z$ be defined by:
\begin{equation}\label{q_def}
q(x)=\{x\ (\textrm{mod } i)\}.
\end{equation}
We have the following simple property:
$$\pi_j\circ q=q_j.$$
As a consequence of this and the structure of cylinder sets, we obtain the following observation, needed later in the description of Bunce-Deddens algebras.
\begin{prop}
The range of $q$ is dense in $ \mathbb Z/N\mathbb Z $. 
 \end{prop}

We denote by $\mathcal E( \mathbb Z/N\mathbb Z)$ the space of locally constant functions on $ \mathbb Z/N\mathbb Z$. This is a dense subspace of the space of continuous functions on $ \mathbb Z/N\mathbb Z$.  For $f \in \mathcal E( \mathbb Z/N\mathbb Z)$, consider the sequence: 
$$a_f(k)= f(q(k)),\ \ k\in\mathbb Z_{\geq 0}.$$ 
Then we have the following observation:

\begin{prop}\label{loc_const}
If $f \in \mathcal E( \mathbb Z/N\mathbb Z)$, then there exists $j \in \mathcal J_N$ such that $a_f(k+j)= a_f(k)$ for every $k\in \mathbb Z_{\geq 0}$. Conversely, if $a(k)$ is a $j$-periodic sequence for some $j \in \mathcal J_N$, then there is a unique $f \in  \mathcal E( \mathbb Z/N\mathbb Z)$ such that $a(k)=a_f(k)$.
\end{prop}

\begin{proof}
The result follows from an observation that any locally constant function on $\mathbb Z/N\mathbb Z$ is a pullback via $\pi_j$ of a function on $\mathbb Z/j\mathbb Z$ for some $j|N$, see \cite{RV}.
\end{proof}

\subsection{BD algebras}
Consider the Hilbert space $\ell^2(\mathbb Z_{\geq 0})$ equipped with the canonical basis $\{E_k\}_{k=0}^{\infty}$. Let   $U: \ell^2(\mathbb Z_{\geq 0}) \rightarrow \ell^2(\mathbb Z_{\geq 0})$ be the unilateral shift given by $UE_k=E_{k+1}$.
 The adjoint of $U$ is given by: 
 \[U^*E_k=\begin{cases}
 E_{k-1} & \textnormal{ if } k\geq 1\\
0 & \textnormal{ if } k=0,
 \end{cases}\]
and we have the relation: 
\begin{equation*}
U^*U=I.
\end{equation*}
 We also use the following diagonal label operator: 
 $$\mathbb K E_k= kE_k.$$ 
 If $\{a(k)\}_{k=0}^\infty$ is a bounded sequence, then $a(\mathbb K)$ is a bounded operator given by: 
 $$a(\mathbb K) E_k= a(k) E_k.$$
 In numerous formulas below we use convention $a(-1)=0$, so that, for example, we have:
 \[a(\K-I)E_k=\begin{cases}
 a(k-1)E_{k} & \textnormal{ if } k\geq 1\\
0 & \textnormal{ if } k=0.
 \end{cases}\]

Given a supernatural number $N$, we define the following algebra of diagonal operators:
\[\mathcal A_{\textrm{diag, per}}(N)= \left\{ a(\mathbb K) : \; a(k) \textnormal{ is $j$-periodic for some $j|N$}\right\}.\]
The norm closure of $\mathcal A_{\textrm{diag, per}}(N)$, denoted by $ A_{\textrm{diag, per}}(N)$, is a commutative unital C$^*$-algebra which, by Proposition \ref{loc_const}, is canonically isomorphic to the C$^*$-algebra of continuous functions on $ \mathbb Z/N\mathbb Z$:
 \begin{equation}\label{diag_id}
\overline{\mathcal A_{\textrm{diag, per}}(N)}=:  A_{\textrm{diag, per}}(N) \cong C( \mathbb Z/N\mathbb Z). 
\end{equation}

\begin{defin}
Given a supernatural number $N$, the {\it Bunce-Deddens-Toeplitz algebra}, denoted by $A(N)$, is the C$^*$-algebra of operators in $\ell^2(\mathbb Z_{\geq 0})$ generated by $U$ and $\mathcal A_{\textrm{diag, per}}(N)$:
\[A(N)= \textnormal{C}^*(U, \mathcal A_{\textrm{diag, per}}(N)).\]
\end{defin}
It is easy to see that for infinite $N$ this definition coincides with the original definition \cite{BD1}, \cite{BD2} given in the introduction.

Let $A_{\textnormal{diag}}(N)$ be  the commutative $^*$-subalgebra of  $A(N)$ consisting of operators diagonal with respect to the canonical basis $\{E_k\}$ of $\ell^2(\mathbb Z_{\geq 0})$.  If the space of sequences which are eventually zero is denoted by $c_{00}$,  we define:
 \[\mathcal A_{\textrm{diag}}(N):= \{a(\mathbb K): a(k)= a_0(k) + a_{\textrm{per}}(k), \; a_0(k)  \in c_{00} \textrm{ and } a_{\textrm{per}}(\mathbb K) \in  \mathcal A_{\textrm{diag, per}}(N)\}\]
which is a separable unital $^*$-algebra.  Some useful properties of this algebra are described in the following statement.

\begin{prop}\label{decomp_prop}
$\mathcal A_{\textrm{diag}}(N)$ is a dense $^*$-subalgebra of $A_{\textnormal{diag}}(N)$. If the space of sequences converging to zero is denoted by $c_0$, then we have the identification:
\[A_{\textnormal{diag}}(N)= \overline{\mathcal A_{\textrm{diag}}(N)}= \{a(\mathbb K): \; a(k)= a_0(k)+a_{\textrm{per}}(k), \  a_0(k)\in c_0,\ a_{\textrm{per}}(k) \in C(\mathbb Z/ N\mathbb Z) \}.
\]
\end{prop}

\begin{proof}
Other than $A_{\textrm{diag, per}}(N)$, the algebra $A_{\textnormal{diag}}(N)$ also contains additional diagonal operators that are in the algebra generated by the unilateral shift $U$. Those are precisely the compact diagonal operators:  $\{a_0(\mathbb K):  \  a_0(k)\in c_0\}$, see \cite{KMR1}. The additive decomposition $a(k)= a_0(k)+a_{\textrm{per}}(k)$ in $\mathcal A_{\textrm{diag}}(N)$ persists in completion $A_{\textnormal{diag}}(N)$ because compact diagonal operators form an ideal in $A_{\textnormal{diag}}(N)$, with the quotient isomorphic to $C(\mathbb Z/ N\mathbb Z)$. In fact, we have the following easy estimate:
\begin{equation*}
\|a(\mathbb K)\|= \|a_0(\mathbb K)+a_{\textrm{per}}(\mathbb K)\|\geq \|a_{\textrm{per}}(\mathbb K)\|,
\end{equation*}
which implies directly the decomposition when passing to limits.
\end{proof}

Let $\mathcal A(N)$ denote the $^*$-algebra generated algebraically by $U, U^*$ and $\mathcal A_{\textrm{diag, per}}(N)$. We have the following description of $\mathcal A(N)$.

\begin{prop}\label{curlA}
$\mathcal A(N)$ is a dense $^*$-subalgebra of $A(N)$. Moreover, we have the following description:
 \[\begin{aligned} \mathcal A(N):= \Big\{ &a\in A(N): \; a= \sum_{n\geq 0} U^n a_{n,0}^+(\mathbb K) + \sum_{n\geq 1}  a_{n,0}^-(\mathbb K)(U^*)^n+ \sum_{n\geq 0} U^n a_{n,\textrm{per}}^+(\mathbb K)\\
&+ \sum_{n\geq 1}  (U^*)^na_{n,\textrm{per}}^-(\mathbb K), \; a_{n,0}^{\pm}(k) \in c_{00}, a_{n, \textrm{per}}^{\pm}(\mathbb K) \in \mathcal A_{\textrm{diag, per}}(N), \textrm{ finite sums} \Big\} .\end{aligned}\]
\end{prop}
 
\begin{proof}
By Proposition 3.1 of  \cite{KMR1} the polynomials in $U$ and $U^*$ which are compact operators are precisely the finite sums of the form:
$$\sum_{n\geq 0} U^n a_{n,0}^+(\mathbb K) + \sum_{n\geq 1}  a_{n,0}^-(\mathbb K)(U^*)^n,$$
where $a_{n,0}^{\pm}(k) \in c_{00}$. They form an ideal in $\mathcal A(N)$ so that, using additionally the commutation relation  \eqref{the_com_rel} below, all the remaining polynomials in $U, U^*$ and $\mathcal A_{\textrm{diag, per}}(N)$ can be written as the last two terms in the statement of the proposition.
\end{proof}

If $a(\mathbb K) \in A_{\textnormal{diag}}(N)$, then we have the commutation relation:
\begin{equation}\label{the_com_rel}
a(\mathbb K)U= Ua(\mathbb K+I).
\end{equation}
In  fact, $A(N)$ is  the  partial crossed product of $ A_{\textnormal{diag}}(N)$ with $\mathbb Z_{\geq 0}$ where the action of $\mathbb Z_{\geq 0}$ on $ A_{\textnormal{diag}}(N)$ is translation by one \cite{E}, \cite{St}. In the trivial case of $N=1$, the algebra $A(1)$ is the Toeplitz algebra, i.e., the C$^*$-algebra generated by $U$. If $N$ is finite, we can also identify $A(N)$ as the tensor product of the Toeplitz algebra with matrices of  size $N \times N$ (see \cite{D} and also Section 4):
\[A(N) \cong A(1) \otimes M_N(\mathbb C). \]

If  $\mathcal K$ are the compact operators in $\ell^2(\mathbb Z_{\geq 0})$, then $\mathcal K$ is an ideal in $A(N)$, and we have the short exact sequence: 
\[0 \rightarrow \mathcal K \rightarrow A(N) \xrightarrow{\xi} B(N) \rightarrow 0\]
where $B(N):= A(N)/\mathcal K$ and $\xi : A(N)\to A(N)/\mathcal K$ is the quotient map.
For any supernatural number $N$, we will call $B(N)$ the {\it Bunce-Deddens algebra}. The Bunce-Deddens algebras are simple for infinite $N$, mutually non-isomorphic and have unique tracial state  \cite{BD1},\cite{BD2},\cite{D}.  

\subsection{Structure of BD algebras}

We now proceed to a more detailed description of the Bunce-Deddens algebras $B(N)$.
Suppose $\{E_l\}_{l\in \mathbb Z}$ is the canonical  basis of $\ell^2(\mathbb Z)$, we let  $V:\ell^2(\mathbb Z) \rightarrow \ell^2(\mathbb Z)$ be the bilateral shift given by: 
$$VE_l=E_{l+1},$$ 
let $\mathbb L$ be the diagonal label operator: 
$$\mathbb L E_l= lE_l,$$  
and let $\mathcal B_{\textnormal{diag}}(N)$ be defined as:
\[\mathcal B_{\textnormal{diag}}(N):= \{b(\mathbb L): \; b(l+j)=b(l) \textnormal{ for some } j\mid N\}.\]
Notice that $B_{\textnormal{diag}}(N):= \overline{\mathcal B_{\textnormal{diag}}(N)}$ is naturally isomorphic to $C( \mathbb Z/N\mathbb Z )$, just like in (\ref{diag_id}).

Similarly to \eqref{the_com_rel} we have the commutation relation:
\begin{equation}\label{the_com_rel2}
b(\mathbb L)V= Vb(\mathbb L+I).
\end{equation}

For any $N$ we  introduce the Toeplitz-like operator $T: \mathcal B(\ell^2(\mathbb Z)) \rightarrow \mathcal B(\ell^2(\mathbb Z_{\geq 0}))$ given by the formula: 
\begin{equation}\label{Toep_def}
T(b)f= Pbf,
\end{equation}
where $f\in \ell^2(\mathbb Z_{\geq 0}) $, and $P: \ell^2(\mathbb Z) \rightarrow \ell^2(\mathbb Z)$ is the orthogonal projection onto the subspace $S=$ span$\{E_l : l\geq 0\}$, which is naturally isomorphic with $\ell^2(\mathbb Z_{\geq 0})$. It is clear that we have:
$$T(I|_{\ell^2(\mathbb Z)})=I|_{\ell^2(\mathbb Z_{\geq 0})}.$$ 
The operator $T$ is a linear, continuous, and $*$-preserving map between the spaces of bounded operators on $\ell^2(\mathbb Z)$ and  $\ell^2(\mathbb Z_{\geq 0})$, and moreover it has the following properties:

\begin{lem} \label{Toep_lemma}
For every $a,b \in \mathcal B(\ell^2(\mathbb Z))$ and any bounded diagonal operator $b (\mathbb L)$:
\begin{enumerate}[(i)]
\item $T(b\,V^n)= T(b) U^n$ and $T(V^{-n}b)= (U^*)^n T(b)$ for  $n \geq 0$
\item $T(a\,b (\mathbb L))= T(a) b (\mathbb K)$
\item $T( b (\mathbb L)\,a)=  b (\mathbb K) T(a)$.
\end{enumerate}
\end{lem}
\begin{proof}
Those statements are obtained via direct calculations. For example, we have:
\begin{equation*}
T(bV^n)f=PbV^nf=Pb PV^nf =T(b) U^n f
\end{equation*}
because for $n\geq 0$ the operator $V^n$ preserves $S$. Other calculations are similar.
\end{proof}
Since any element in C$^*(V, \mathcal B_{\textnormal{diag}}(N))$ can be approximated by a finite sum of the form:
\begin{equation}\label{pol_In_B(N)}
\sum_{n \in \mathbb Z} V^nb_{n} (\mathbb L), 
\end{equation}
with $b_{n} (\mathbb L) \in \mathcal B_{\textnormal{diag}}(N)$, it is clear that $T$  maps C$^*(V, \mathcal B_{\textnormal{diag}}(N))$ into $A(N)$. 

\begin{prop}\label{iden_2}
For any supernatural number $N$ the algebras $B(N)$ and C\,$^*(V, \mathcal B_{\textnormal{diag}}(N))$ are isomorphic.
\end{prop}

\begin{proof}
For any $b_1, b_2 \in $ C$^*(V, \mathcal B_{\textnormal{diag}}(N))$, it can be shown just like for regular Toeplitz operators, that: 
\[T(b_1b_2)= T(b_1)T(b_2) + K\]
for some compact operator $K \in \mathcal K$. Now, the map 
\[[T]: \textnormal{C}^*(V, \mathcal B_{\textnormal{diag}}(N))\rightarrow A(N)/\mathcal K\]
 defined by:
 \[[T](b)= T(b) + \mathcal K\]
gives the required isomorphism. 
\end{proof}
Let $\mathcal B(N)$ be the $^*$-algebra generated algebraically by $V, V^{-1}$ and $\mathcal B_{\textnormal{diag}}(N)$. Notice that we have:
$$\mathcal B(N)=  \mathcal A(N)/(\mathcal A(N) \cap \mathcal K),$$
i.e. $\mathcal B(N)$ is the image of $\mathcal A(N)$ under the quotient map $\xi$. Also, because of the commutation relation \eqref{the_com_rel2}, the elements of $\mathcal B(N)$ are precisely the finite sums of the form given in \eqref{pol_In_B(N)}.

We have the following further identification of $B(N)$, see \cite{E}.
 \begin{prop}\label{cross_iden}
 For infinite $N$ the algebra $B(N)$ can be identified with the crossed product of $C(\mathbb Z/N\mathbb Z)$ with $\mathbb Z$, acting on $C(\mathbb Z/N\mathbb Z)$ via shifts. i.e.,
 \[B(N) \cong C(\mathbb Z/N\mathbb Z) \rtimes_{\sigma} \mathbb Z\] where for $f \in C(\mathbb Z/N\mathbb Z)$,  $\sigma f(x)= f(x+1)$.
  \end{prop}

For finite $N$ one can identify $B(N)$ with $C(S^1)  \otimes M_N(\mathbb C)$. This is useful for the purpose of classifying derivations in $A(N)$ and $B(N)$  in the next section. We describe this identification in detail below. 

 \begin{prop}\label{iden_3}
 For a finite supernatural number $N$ there is an isomorphism:
$$\textnormal{C}^*(V, \mathcal B_{\textnormal{diag}}(N)) \cong C(S^1) \otimes  M_N(\mathbb C).$$
 \end{prop}
 
 \begin{proof}
 We first relabel the basis elements of $\ell^2(\mathbb Z)$ as follows:
 $$\{E_{kN+j}\; \vert \; k \in \mathbb Z, 0 \leq j <N\}.$$ 
 Consider the following sequence: 
 \begin{equation}\label{eN_def}
e_N(l)=\begin{cases}
 1 & \textnormal{ if } N \mid l\\
 0 & \textnormal{otherwise}.
 \end{cases} 
\end{equation}
 Then clearly we have periodicity
 $$e_N(l+N)=e_N(l),$$ 
 and the following formula: 
 $$e_N(\mathbb L) E_{kN+j}= \delta_{j,0} E_{kN+j}.$$ 
 For $0 \leq s,r < N$, we define the operators: 
 $$P_{sr}:= V^s e_N(\mathbb L) V^{-r}.$$ 
 It is easy to verify using the above formulas that $P_{sr}$ have the following properties:
 \begin{enumerate}[(i)]
 \item $P_{sr}^*= P_{rs}$
 \item $P_{sr} P_{tq}= \delta_{tr}P_{sq}$.
 \end{enumerate}
 As a consequence, 
 if $E_{sr}$ are the standard basis elements of $M_N(\mathbb C)$, then the map $P_{sr} \mapsto E_{sr}$ induces the isomorphism C$^*(P_{sr}) \cong M_N(\mathbb C)$. Moreover, any element of $\mathcal B_{\textnormal{diag}}(N)$ can be written as a linear combination of $P_{rr}$, $0 \leq r < N$. We also have the relation:
 \[V= P_{10}+P_{21}+ \cdots+ P_{(N-1)(N-2)} + V^N P_{0(N-1)},\]
 which can be verified by a direct calculation on basis elements.
 Therefore, we obtain:
 $$\textnormal{C}^*(V, \mathcal B_{\textnormal{diag}}(N)) \cong {\textnormal{C}}^*(P_{sr}, V^N).$$ 
 Consequently, because $V^N$ commutes with the operators 
 $P_{sr}$ for all $0 \leq s,r < N$, we have:
 \[\textnormal{C}^*(V, \mathcal B_{\textnormal{diag}}(N)) \cong  \textnormal{C}^*(V^N) \otimes \textnormal{C}^*(P_{sr}) \cong C(S^1) \otimes M_N(\mathbb C).\]
  Here C$^*(V^N)$ is isomorphic with $C(S^1)$ because $V^N$ is equivalent to the usual bilateral shift.
 \end{proof}


\section{Covariant Derivations}

\subsection{Derivations.}

A {\it derivation} $d$ in $A(N)$ with domain $\mathcal A(N)$ is a linear map $d: \mathcal A(N) \rightarrow A(N)$ which satisfies the Leibniz rule: 
\[d(ab)= d(a)b+ ad(b)\]
 for all $a,b \in \mathcal A(N)$. In this paper we only study derivations $d$ with domain $\mathcal A(N)$, and derivations $\delta$ in $B(N)$ with domain $\mathcal B(N)$, so we will not explicitly mention domains below. 
 
 A derivation $d$ is called {\it approximately inner} if there are $a_n\in A(N)$  such that 
 $$d(a) = \lim_{n\to\infty}[a_n,a]$$ 
 for $a\in\mathcal A(N)$.
 
The first important observation is that any derivation in $A(N)$ preserves compact operators.
 
 \begin{theo}
 If $d: \mathcal A(N) \rightarrow  A(N)$ is a derivation, then $d: \mathcal A(N)\cap\mathcal K \rightarrow \mathcal K$.
 \end{theo} 
 
 \begin{proof}
 It is enough to prove that $d(P_0)$ is compact, where $P_0$ is the orthogonal projection onto the one-dimensional subspace spanned by $E_0$, because $\mathcal A(N)\cap\mathcal K$ is comprised of linear combinations of expressions of the form: $U^rP_0(U^*)^s$.  The result then follows immediately from the Leibniz property. To see that $d(P_0)$ is compact, simply apply $d$ to both sides of the relation $P_0=P_0^2$ to obtain:
 $$d(P_0)=d(P_0)P_0+P_0d(P_0)\in \mathcal K,$$
 which completes the proof.
 \end{proof}
 
 As a consequence of the above theorem, if $d:\mathcal A(N) \rightarrow A(N)$ is a derivation in $A(N)$, then $[d]: \mathcal B(N) \rightarrow B(N)$ defined by 
\[[d](a+\mathcal K) := da+ \mathcal K\]
 gives a derivation in $B(N)$ where, as before, $\mathcal B(N)=  \mathcal A(N)+ \mathcal K$.


\subsection{Classification of covariant derivations.}

 For each $\theta \in [0, 2\pi)$, let $\rho^{\mathbb K}_{\theta}: A(N) \rightarrow A(N)$ be  defined by: 
\[\rho^{\mathbb K}_{\theta} (a)= e^{i\theta \mathbb K} a e^{-i\theta \mathbb K}.\]
Then we have:
\begin{equation}\label{rho_action}
\rho^{\mathbb K}_{\theta}(U)= e^{i\theta}U,\ \rho^{\mathbb K}_{\theta}(U^*)= e^{-i\theta}U^* \textrm{ and } \rho^{\mathbb K}_{\theta}(a(\mathbb K))= a(\mathbb K).
\end{equation}
Thus, $\rho^{\mathbb K}_{\theta}$ is a well-defined automorphism of $A(N)$, and $\rho^{\mathbb K}_{\theta}$ preserves $\mathcal A(N)$.

 \begin{defin}
 Given $n \in \mathbb Z$, a derivation $d$ in $A(N)$ is said to be a {\it $n$-covariant derivation} if the relation 
 $$(\rho^{\mathbb K}_{\theta})^{-1}d(\rho^{\mathbb K}_{\theta}(a))= e^{-in\theta} d(a)$$ holds.
 \end{defin}
 
 Similarly, for $\theta \in [0, 2\pi)$,  we let $\rho^{\mathbb L}_{\theta}$ be  the automorphism of $B(N)$, preserving $\mathcal B(N)$, defined by: 
\[\rho^{\mathbb L}_{\theta}(b)= e^{i\theta \mathbb L} b e^{-i\theta \mathbb L}.\]

 \begin{defin}
 Given $n \in \mathbb Z$, a derivation $\delta$ in $B(N)$ is said to be a {\it $n$-covariant derivation} if the relation 
 $$(\rho^{\mathbb L}_{\theta})^{-1}\delta(\rho^{\mathbb L}_{\theta}(b))= e^{-in\theta} \delta(b)$$ holds.
 \end{defin}

An important step in classifying derivations on Bunce-Deddens-Toeplitz algebras is the classification the $n$-covariant derivations in $A(N)$ since they arise as Fourier coefficients of general derivations. First we establish the following useful description of covariant subspaces in $A(N)$.

\begin{prop}
We have the following equality:
\[A_{\textnormal{diag}}(N)= \{a\in A(N): \; \rho^{\mathbb K}_{\theta}(a)=  a\}.\]
\end{prop}

\begin{proof}
Clearly if $a\in A_{\textnormal{diag}}(N)$, then $ \rho^{\mathbb K}_{\theta}(a)=  a$ by formula (\ref{rho_action}).
Conversely, if $a\in A(N)$ satisfies $ \rho^{\mathbb K}_{\theta}(a)=  a$ then the equation:
$$(E_k, e^{i\theta \mathbb K} a e^{-i\theta \mathbb K} E_l)= (E_k, aE_l)$$ 
implies that have:
$$e^{i\theta (k-l)} (E_k, aE_l)= (E_k, aE_l)$$ 
for every $\theta \in [0, 2\pi)$ and every $k,l$, from which it follows that $a$ is a diagonal operator.
\end{proof}

\begin{prop}\label{n_spec_subsp}
Denote by $A_n(N)$ the $n$-th spectral subspace of $\rho^{\mathbb K}_{\theta}$:
\[A_n(N):= \{a\in A(N): \; \rho^{\mathbb K}_{\theta}(a)= e^{i n\theta} a\}.\]
Then we have: 
 \[A_n(N)= \begin{cases}
\{U^na(\mathbb K): \; a(\mathbb K)\in A_{\textnormal{diag}}(N)\} & \textnormal{ if } n\geq 0\\
\{a(\mathbb K)(U^*)^{-n}: \; a(\mathbb K)\in A_{\textnormal{diag}}(N)\} & \textnormal{ if } n<0.
\end{cases}\] 
\end{prop}

\begin{proof}
We will give the proof for $n>0$; the proof for $n < 0$ works similarly. 

Since we have:
$$\rho^{\mathbb K}_{\theta}(U^n a(\mathbb K))=e^{in\theta}U^n a(\mathbb K),$$ 
one containment clearly follows. Conversely, if $a\in A_n(N)$ then $a(U^*)^n \in A_{\textnormal{diag}}(N)$, hence is of the form $a(U^*)^n= a(\mathbb K)$ by the previous proposition. Consequently, we have:
$$a=a(\mathbb K)U^n = U^n a(\mathbb K+nI),$$
which shows the other containment.
\end{proof}

It turns out that $n$-covariant derivations in $A(N)$ can be described explicitly.

\begin{theo} \label{n_cov_der_formula}
If $d$ is an $n$-covariant derivation in $A(N)$, then there exists a diagonal operator $\beta_n(\mathbb K)$  such that $d$ can be written as:
 \begin{equation}\label{d_com_formulas}
d(a)=\begin{cases}
 [U^n\beta_n(\mathbb K), a] &  \textnormal{ if } n\geq 0\\
 [\beta_n(\mathbb K)(U^*)^{-n}, a] &  \textnormal{ if } n< 0,
 \end{cases}
  \end{equation}
where the operator $\beta_n(\mathbb K)$ satisfies the following conditions: if $N$ is infinite and $n \neq 0$ or $N$ is finite but $N \nmid n$, then 
$$\beta_n(\mathbb K) \in A_{\textnormal{diag}}(N),$$ (so in particular it is bounded); otherwise:
$$\beta_n(\mathbb K)-\beta_n(\mathbb K-I)\in A_{\textnormal{diag}}(N).$$
The operator $\beta_n(\mathbb K)$ is unique except when $n=0$ where $\beta_0(\mathbb K)$ is unique up to an additive constant. Conversely, given any $\beta_n(\mathbb K)$ satisfying those properties, the formulas above define $n$-covariant derivations in $A(N)$.

\end{theo}

\begin{proof}
Suppose $n>0$ and $d$ is a $n$-covariant derivation in $A(N)$. It follows that we have $d(a(\mathbb K)) \in A_n(N)$, and hence the formula:
 \begin{equation}\label{d_tilde}
 d(a(\mathbb K))= U^n \tilde d(a(\mathbb K))
 \end{equation}
for some $\tilde d(a(\mathbb K)) \in A_{\textnormal{diag}}(N)$ by Proposition \ref{n_spec_subsp}.

Similarly, there exists $\alpha_n(\mathbb K) \in  A_{\textnormal{diag}}(N)$ such that:
 \begin{equation}\label{d_on_U_U*}
 \begin{aligned}  d(U^*)&= -U^{n-1}\alpha_n(\mathbb K)  \;\;\;\textnormal{ and }\\
d(U)&=  U^{n+1}\alpha_n(\mathbb K+I),
 \end{aligned}\end{equation}
 where the last equation follows from the relation $d(U^*)U+U^*d(U)=0$. 

From formula \eqref{d_tilde} for every $a(\mathbb K), b(\mathbb K) \in \mathcal A_{\textnormal{diag}}(N)$ we have the following: 
\[\begin{aligned} \tilde d(a(\mathbb K)b(\mathbb K))&= (U^*)^n d(a(\mathbb K)b(\mathbb K)) \\
&=(U^*)^nd(a(\mathbb K)) b(\mathbb K)  + (U^*)^n a(\mathbb K)d(b(\mathbb K)) \\
&=\tilde d(a(\mathbb K))b(\mathbb K) + a(\mathbb K +nI) \tilde d(b(\mathbb K)).  \end{aligned}\] 
Since $\tilde d(a(\mathbb K)b(\mathbb K))= \tilde d(b(\mathbb K)a(\mathbb K))$, it follows that: 
\begin{equation}\label{a-bequ}
\tilde d(a(\mathbb K))[b(\mathbb K )-b(\mathbb K +nI)]= \tilde d(b(\mathbb K))[a(\mathbb K )-a(\mathbb K +nI)]. 
\end{equation}
 For given $n$ we can always choose $a(\mathbb K)$ such that $a(k )-a(k+n) \neq 0$ for every $k$. Using such $a(\mathbb K)$ we define:
 $$\beta_n(\mathbb K)= \tilde d(a(\mathbb K))(a(\mathbb K )-a(\mathbb K +nI))^{-1},$$ 
 which is independent of the choice of $a$ by the formula (\ref{a-bequ}). It follows that:
 \begin{equation}\label{d_on_a(K)}
d(a(\mathbb K))=  U^n\beta_n(\mathbb K) [a(\mathbb K )-a(\mathbb K +nI)]
\end{equation}
 for any $a(\mathbb K)$ because, when $a(k +n)-a(k) = 0$ for some $k$, then both sides of the above equation are zero, as implied again by the formula (\ref{a-bequ}).
  
Next, applying $d$ to the commutation relation $U^*a(\mathbb K)=a(\mathbb K +I)U^*$ we obtain:
\[(\beta_n(\mathbb K ) - \beta_n(\mathbb K-I) - \alpha_n(\mathbb K)) [a(\mathbb K )-a(\mathbb K +nI)]=0.\]
It follows that we must have:
\[ \alpha_n(\mathbb K) =\beta_n(\mathbb K ) - \beta_n(\mathbb K-I).\]
This leads to formulas:
 \begin{equation*}
 \begin{aligned}  d(U^*)&= -U^{n-1}(\beta_n(\mathbb K ) - \beta_n(\mathbb K-I) ) \;\;\;\textnormal{ and }\\
d(U)&=  U^{n+1}(\beta_n(\mathbb K + I) - \beta_n(\mathbb K)),
 \end{aligned}\end{equation*}
and so we have that $d(a)= [U^n\beta_n(\mathbb K), a]$ holds true for all the generators and hence for every $a\in \mathcal A(N)$ and $n>0$. Notice also that we can compute $\beta_n$ in terms of $\alpha_n$ by the formula:
\[\beta_n(k)= \sum_{i=0}^{k}\alpha_{n}(i),\] 
The proof for $n<0$ works similarly.

If $n=0$,  the formulas for $d(U)$ and $d(U^*)$ are:
\[d(U)=U\alpha_0(\mathbb K), \;\;\; d(U^*)=-\alpha_0(\mathbb K) U^*.\]
We claim that in this case we have: 
$$d(a(\mathbb K))=0.$$ 
This is because for an invariant derivation we have: 
$$d:\mathcal A_{\textnormal{diag}}(N)\to A_{\textnormal{diag}}(N),$$ 
and elements of $\mathcal A_{\textnormal{diag}}(N)$ are finite linear combinations of diagonal orthogonal projections. If $P\in \mathcal A_{\textnormal{diag}}(N)$ is such a projection, by applying $d$ to $P^2=P$ we obtain: 
$$(I-2P)d(P)=0,$$ 
which implies $d(P)=0$. 

To obtain $d(a)=[\beta_0(\K),a]$, we define the operator $\beta_0(\mathbb K)$ as the solution of the equation:
$$\alpha_0(\mathbb K)=\beta_0(\mathbb K+I)-\beta_0(\mathbb K),$$ 
and so it is determined only up to additive constant. We usually make a particular choice:
\[\beta_0(k)= \sum_{i=0}^{k-1}\alpha_{0}(i),\] 
with $\beta_0(0)=0$.

There are additional restrictions on $\beta_n(\mathbb K)$. For $N$ infinite and $n \neq 0$,  we must have 
\begin{equation*}
 \tilde d(a(\mathbb K))=\beta_n(\mathbb K) [a(\mathbb K )-a(\mathbb K +nI)] \in A_{\textnormal{diag}}(N)
\end{equation*}
  for every $a(\mathbb K) \in \mathcal A_{\textnormal{diag}}(N)$. By choosing for example $a(k)= e^{2\pi ik/l}$, an $l$-periodic sequence, with $l\mid N$ but $l \nmid n$, it is clear that $a( k +n)-a( k) \neq 0$ for every $k$ and so the operator $a(\K +n)-a(\K)$ is invertible. It follows that $\beta_n(\mathbb K)$ must belong to $A_{\textnormal{diag}}(N)$.

 If $N$ is finite and $N\nmid n$ then we can choose an $N$-periodic sequence and argue as above to show that $\beta_n(\mathbb K) \in A_{\textnormal{diag}}(N)$.
 
Conversely, given any $\beta_n(\mathbb K)$ satisfying the properties and derivation $d$ given by the commutator formula (\ref{d_com_formulas}), the expressions for $d$ on generators (\ref{d_on_U_U*}), (\ref{d_on_a(K)}) imply that $d$ is a well defined derivation in $A(N)$.
In particular, if $N$ is finite, $N\mid n$, and $a(k)$ is any $l$-periodic sequence where $l\mid N$ is also $n$-periodic, then $a(k )-a(k+n)=0$. Moreover, if $a(k)$ a sequence that is eventually zero, then for large enough $k$ we have $a(k )-a(k+n)=0$. Thus $ \tilde d(a(\mathbb K))$ is eventually zero for every $k$, and there are no additional restrictions on $\beta_n(\mathbb K)$ in this case.
\end{proof}

From this theorem it is clear that if $N$ is infinite and $n \neq 0$ or $N$ is finite but $N \nmid n$, then the $n$-covariant derivation $d$ in $A(N)$  is inner. Otherwise such an $n$-covariant derivation $d$ is in general not inner. 

We simultaneously state here without a detailed proof a similar classification of $n$-covariant derivations on Bunce-Deddens algebras. In fact, all of the arguments in the unilateral shift case of Theorem \ref{n_cov_der_formula} work the same (if not simpler) in the bilateral case needed for Bunce-Deddens algebras.

\begin{theo}\label{BDncov}
If $\delta$ is an $n$-covariant derivation in $B(N)$, then there exists  $\eta_n(\mathbb L)$ such that 
\begin{equation*}
\delta(a)=[V^n\eta_n(\mathbb L), a]
\end{equation*}
for every $a$ in $\mathcal B$.
If $N$ is infinite and $n \neq 0$ or $N$ is finite but $N \nmid n$ then: 
$$\eta_n(\mathbb L)\in B_{\textnormal{diag}}(N),$$ 
otherwise we have: 
$$\eta_n(\mathbb L+I)-\eta_n(\mathbb L)\in B_{\textnormal{diag}}(N).$$ 
Conversely, given any $\eta_n(\mathbb L)$ satisfying those properties, the formulas above define $n$-covariant derivations in $B(N)$.
\end{theo}

\subsection{Properties of covariant derivations}

In general, if an $n$-covariant derivation is approximately inner then it can also be approximated by inner $n$-covariant derivations.
\begin{prop}\label{CovAppProp}
Suppose $d$ is an $n$-covariant derivation in $A(N)$. If $d$ is approximately inner, then there exists a sequence of $n$-covariant derivations $\{d^M\}$ in $A(N)$ such that, for every $a \in \mathcal A(N)$ we have:
\[d(a)= \lim_{M \rightarrow \infty} d^M(a). \]
\end{prop}
\begin{proof}
Given an element $a\in A(N)$, define its $\rho^{\mathbb K}_{\theta}$ $n$-th Fourier component by:
 \[(a)_n=  \frac 1{2\pi} \int_0^{2\pi} e^{-in\theta} \rho^{\mathbb K}_{\theta}(a) d\theta.\]
 If $d$ is approximately inner then there is a sequence $\{z^M\}$, with $z^M \in A(N)$ such that: 
 $$d(a)= \lim \limits_{M \rightarrow \infty} [a, z^M]$$ 
 for every $a \in \mathcal A(N)$. It can be easily checked that the Fourier component $(z^M)_n$ is in $A_{n}(N)$. So, it is sufficient to show that: 
 \[d(a)= \lim_{M \rightarrow \infty} [a, (z^M)_n]\]
 on the generators $U$, $ U^*$ and $a(\K)$ as the result then follows from Proposition \ref{n_spec_subsp}.
 Since 
 $$Uz^M- z^M U \rightarrow d(U),$$ 
 we equivalently have: 
 $$e^{-in \theta} (z^M -U^*z^M U) \rightarrow e^{-in\theta} U^* d(U).$$ 
 So, given $\epsilon >0$, there is an integer $m$ such that for every $M \geq n$, we can estimate:
 \[\|e^{-in\theta} (z^M -U^*z^M U - U^*d(U)) \| < \epsilon .\]
Since $\rho^{\mathbb K}_{\theta}(U^* d(U))= e^{in\theta} U^*d(U)$, the estimate can be written as:
\[\begin{aligned} \|e^{-in\theta} \rho^{\mathbb K}_{\theta}(z^M) - e^{-in\theta}U^*\rho^{\mathbb K}_{\theta}(z^M) U - U^*d(U) \| = \|e^{-in\theta} \rho^{\mathbb K}_{\theta}(z^M - U^*z^M U - U^*d(U)) \| < \epsilon.
  \end{aligned}\]
 Consequently, we have:
 \[\begin{aligned}&\| (z^M)_n - U^* (z^M)_nU - U^*d(U) \| \leq \\
 &\leq\frac 1{2\pi} \int_0^{2\pi} \|e^{-in\theta} \rho^{\mathbb K}_{\theta}(z^M) - e^{-in\theta}U^*\rho^{\mathbb K}_{\theta}(z^M) U - U^*d(U) \| d\theta < \epsilon,\end{aligned}\]
 thus proving  the convergence:
 $$d(U)= \lim_{M \rightarrow \infty} [U, (z^M)_n].$$ 
 A similar proof works for $d$ acting on $U^*$ and $a(\K)$, so the result follows.
\end{proof}

The explicit formulas of Theorem \ref{n_cov_der_formula} allow us to discuss when an $n$-covariant derivation is approximately inner. There are several cases to consider. When $N$ is infinite we separately consider the case when $n=0$, while when $N$ is finite, there are differences depending on whether $n$ is a multiple or $N$ or not.

First consider an invariant derivation $d$ in $A(N)$ given by
\[d(U)=U\alpha_0(\mathbb K), \;\;\; d(U^*)=-\alpha_0(\mathbb K) U^*,\;\;\; d(a(\mathbb K))=0.\]

\begin{lem}\label{inv_app_inn_c_0}
Suppose $d$ is an invariant derivation in $A(N)$ with $N$ infinite. If $\alpha_0(k) \in c_0$ then $d$ is approximately inner.
\end{lem}
\begin{proof} Like in \cite{KMR1} we define $\alpha_{0}^M (\mathbb K) \in A_{\textnormal{diag}}(N)$ by:
\[\alpha_{0}^M (k)= \begin{cases}
\alpha_{0}(k) &  \textnormal{ if } k\leq M\\
0 & \textnormal{otherwise}.
\end{cases}\]
Then we see that $ \alpha_{0}^M (\mathbb K) $ converges to $\alpha_{0}(\mathbb K)$ in norm as $M$ tends to infinity because:
\[\| \alpha_{0}(\mathbb K) - \alpha_{0}^M (\mathbb K) \| = \sup_{k} |\alpha_{0}(k) - \alpha_{0}^M (k)|= \sup_{k>M} |\alpha_{0}(k)| \xrightarrow{M \rightarrow \infty} 0.\]
The sequence $\beta_{0}^M (k)$ defined by:
\[\beta_{0}^M (k):= \sum_{j=0}^{k-1} \alpha_{0}^M (j)\]
is eventually constant and in particular, it is bounded. Therefore, we have that
$$d^M(a):=  [\beta_{0}^M(\mathbb K), a]$$
 is an inner derivation. To prove that $d^M(a) \rightarrow d(a)$ as $M \rightarrow \infty$ for every $a\in  \mathcal A(N)$, it is enough to check that on the generators $U$ and  $U^*$. But this follows easily since we have:
\[
d(U)- d^M(U)= U(\alpha_{0}(\mathbb K) - \alpha_{0}^M (\mathbb K)),\]
and similarly for $U^*$.
Thus $d$ is an approximately inner derivation. 
\end{proof}

The second case of invariant approximately inner derivations is described next.

\begin{lem}\label{inv_app_inn_per}
Suppose $d$ is an invariant derivation in $A(N)$ with an infinite supernatural number $N$. If $\alpha_0(k)=f(q(k))$ where $f \in C( \mathbb Z/N\mathbb Z)$ with $\int_ {\mathbb Z/N\mathbb Z} f(x) d_Hx=0$ and $q$ is the quotient map introduced in \eqref{q_def},  then $d$ is approximately inner.
\end{lem}
\begin{proof}
If $f \in C( \mathbb Z/N\mathbb Z)$, then 
$$f(x)= \lim_{M \rightarrow \infty} f^M(x)$$ 
uniformly for some sequence of locally constant functions $f^M(x)$ on $ \mathbb Z/N\mathbb Z$. By Proposition \ref{loc_const} there is a sequence of numbers $\{j^M\}$,  such that $j^M \mid N$ and for every $M$ the sequence $f^M(q(k))$ is $j^M$-periodic.  Moreover, by subtracting constants: 
$$\int_{\mathbb Z/N\mathbb Z} f^M(x) d_Hx \xrightarrow{M \rightarrow \infty} 0,$$
if necessary, we can choose $f^M$ so that: 
\begin{equation}\label{zero_int}
\int_{\mathbb Z/N\mathbb Z} f^M(x) d_Hx=0.
\end{equation}

Now consider the sequence $\{\alpha_0^M(k)\}:=\{f^M(q(k))\} $. A simple calculation shows that the equation (\ref{zero_int}) is equivalent to the following condition:
$$\sum_{i=0}^{j^M-1}\alpha_0^M(i)=0.$$
Furthermore, defining
\[\beta_0^M(k)= \sum_{i=0}^{k-1}\alpha_0^M(i),\] 
we have that $\beta_0^M$ is  also $j^M$-periodic because:
\[\beta_0^M(k+j^M)= \sum_{i=0}^{k-1}\alpha_0^M(i) +  \sum_{i=k}^{k+j^M-1}\alpha_0^M(i)= \sum_{i=0}^{k-1}\alpha_0^M(i) + \sum_{i=0}^{j^M-1}\alpha_0^M(i)=\beta_0^M(k).\]

Let $d^M: \mathcal A(N) \rightarrow  A(N)$ be the derivation defined by:
\[d^M(U)=U\alpha_0^M(\mathbb K), \;\;\; d^M(U^*)=-\alpha_0^M(\mathbb K) U^*,\;\;\; d^M(a(\mathbb K))=0.\]
Thus, we have:
$$d^M(a)= [a, \beta_0^M(\mathbb K) ]$$ 
for every $a\in  \mathcal A(N)$ and, since $\beta_0^M(\mathbb K) \in A_{\textnormal{diag}}(N)$, it follows that $d^M$ is an inner derivation. Moreover, the sequence $\{d^M\}$ approximates $d$ because:
\[\|d(U)- d^M(U)\|= \sup_k |\alpha_0(k)- \alpha_0^M(k)|= \sup_k |f(q(k))- f^M(q(k))| \rightarrow 0\]  as $M \rightarrow \infty$. Similarly, we obtain that:
$$\lim_{M \rightarrow \infty} d^M(U^*)= d(U^*),$$ 
and therefore, $d$ is approximately inner.
\end{proof}

Now we consider the case of finite $N$ and $N \mid n$. Further examples of approximately inner $n$-covariant derivations are described by the following lemma.

\begin{lem}\label{n_cov_app_inn}
Suppose $d$ is an $n$-covariant derivation in $A(N)$ where $N$ is finite and $N \mid n$. Then $d$ is approximately inner if $\alpha_n(k) \in c_0$.
\end{lem}

\begin{proof}
 The proof is essentially the same as that of Lemma \ref{inv_app_inn_c_0}.
\end{proof}

To complete the classification of $n$-covariant derivations we introduce special derivations $d_{n,\mathbb K}$ in $A(N)$ given by:
\[d_{n,\mathbb K}(a)=\begin{cases}
[U^n(\mathbb K +I), a] &  \textnormal{ if } n\geq 0\\
[(\mathbb K + I)(U^*)^n, a] &  \textnormal{ if } n<0.
\end{cases}\] 
Notice that by Theorem \ref{n_cov_der_formula}, derivations $d_{n,\mathbb K}(a)$ are well-defined in $A(N)$ when $N$ is infinite and $n = 0$ or when $N$ is finite and $N \mid n$, because we have the following relation for the diagonal operator coefficients: 
$$(\mathbb K+I)-\mathbb K=I\in A_{\textnormal{diag}}(N),$$
and there are no other restrictions on the coefficients for those cases.

\begin{theo}\label{ncov_N|n}
If $d$ is an $n$-covariant derivation in $A(N)$ where $N$ is infinite and $n = 0$ or when $N$ is finite and $N \mid n$, then there exists a unique constant $C_n$ such that 
\[d(a)= C_n d_{n,\mathbb K}(a) + \tilde d (a)\] for every $a\in \mathcal A(N)$, where $\tilde d$ is an approximately inner derivation.
\end{theo}

\begin{proof}
Consider the case $n>0$ and finite $N$. We then have the formula: $d(a)=[U^n\beta_n(\mathbb K), a]$, and the condition: $\alpha_n(\mathbb K) =\beta_n(\mathbb K +I) - \beta_n(\mathbb K) \in A_{\textnormal{diag}}(N)$.
We apply Proposition \ref{decomp_prop} to $ \alpha_n(k)$, and refine it in the following way:
\[\alpha_n(k)=\alpha_{n,0}(k) +C_n +  \alpha_{n, \textnormal{per}}(k)\]
where $C_n$ is a constant, $\alpha_{n,0}(k) \in c_0$, $  \alpha_{n, \textnormal{per}}(k+N)=\alpha_{n, \textnormal{per}}(k)$ and 
$$\sum_{k=0}^{N-1}  \alpha_{n, \textnormal{per}}(k) =0.$$ 
We then decompose $\beta_n$ using the following:
\[\beta_{n,0}(k):= \sum_{j=0}^{k}\alpha_{n,0}(j), \;\;\; \beta_{n, \textnormal{per}}(k):=  \sum_{j=0}^{k}\alpha_{n,\textnormal{per}}(j).\]
It is easy to verify that $ \beta_{n, \textnormal{per}}(k)$ is $N$-periodic, just as in Lemma \ref{inv_app_inn_per}. We then obtain:
\[\beta_n(k)=\beta_{n,0}(k) + C_n(k +1)+  \beta_{n, \textnormal{per}}(k).\]
So, for $n>0$, the derivation $d$ decomposes as follows:
\[d(a)= [U^n\beta_{n,0}(\mathbb K), a] + C_nd_{n,\mathbb K}(a) + [U^n  \beta_{n, \textnormal{per}}(\mathbb K), a].   \]

We know that $ [U^n\beta_{n,0}(\mathbb K), a]$ is approximately inner by Lemma \ref{n_cov_app_inn}. Moreover, since $ \beta_{n, \textnormal{per}}(\mathbb K) \in A(N)$,   $[U^n  \beta_{n, \textnormal{per}}(\mathbb K), a]$ is an inner derivation. 
To conclude the theorem for $n>0$, and verify the uniqueness, it only remains to show that $d_{n,\mathbb K}(a) $ is not approximately inner. This easily follows from the methods of Theorem 4.4 in \cite{KMR1}, in the following way.

Assume to the contrary that $d_{n,\mathbb K}$ is approximately inner. By Proposition \ref{CovAppProp} there exists a sequence $\mu^M(\K)\in A_{\textnormal{diag}}(N)$, $M=1,2,\ldots$, such that: 
\begin{equation*}
d_{n,\mathbb K}(a) = \lim_{M\to\infty} [U^n\mu^M(\K),a]
\end{equation*}
for all $a\in\mathcal{A}$. In particular, we must have:
\begin{equation*}
d_{n,\mathbb K}(U) = U^{n+1} = \lim_{M\to\infty} U^{n+1}(\mu^M(\K+I)-\mu^M(\K)).
\end{equation*}
Without loss of generality assume $\mu^M(k)$ are real, or else in the argument below simply consider the real part of $\mu^M(k)$.  
The above equation implies that:
\begin{equation*}
\lim_{M\to\infty}\underset{k}{\textrm{sup}}|(\mu^M(k+1)-\mu^M(k)) - 1)| = 0.
\end{equation*}
Therefore for any small $\varepsilon>0$ there are $k$ and $m$ large enough so that we have: 
\begin{equation*}
1-\varepsilon \le \mu^M(k+1) - \mu^M(k)\le 1+ \varepsilon.
\end{equation*}
By telescoping $\mu^M(k)$, we get:
\begin{equation*}
\mu^M(k) = (\mu^M(k) - \mu^M(k-1)) + \cdots + (\mu^M(k_0+1) - \mu^M(k_0)) + \mu^M(k_0)
\end{equation*}
for some fixed $k_0$.  Together the last two formulas imply that:
$$\mu^M(k)\ge (1-\varepsilon)(k-k_0) + \mu^M(k_0),$$
 which goes to infinity as $k$ goes to infinity.  This contradicts the fact that $\mu^M(\K)\in A_{\textnormal{diag}}(N)$ which completes the proof for $n>0$ and finite $N$.
Cases $n=0$ and $n<0$ can be proved very similarly. 

\end{proof}

We summarize the remaining cases of our classification of $n$-covariant derivations in $A(N)$ in the next theorem.

\begin{theo}\label{n_cov_inner}
Suppose $d$ is an $n$-covariant derivation in $A(N)$. If $N$ is infinite with $n \neq 0$ or if $N$ is finite with $N\nmid n$, then $d$ is an inner derivation.
\end{theo}

\begin{proof}
From Theorem \ref{n_cov_der_formula} we already know that $\beta_n(\mathbb K) \in A_{\textnormal{diag}}(N)$ when  $N$ is infinite with $n \neq 0$ or when $N$ is finite with $N\nmid n$. Thus $d$ is an inner derivation.
\end{proof}

This concludes the classification of $n$-covariant derivations in $A(N)$.  Classification  of $n$-covariant derivations in $B(N)$ is somewhat simpler and can be obtained by applying the same methods as used in the classification of $n$-covariant derivations in $A(N)$. 

\begin{theo}
If $\delta$ is a $n$-covariant derivation in $B(N)$ where $N$ is infinite and $n \neq 0$ or $N$ is finite but $N \nmid n$ then $\delta$ is an inner derivation. Otherwise there exists a unique constant $C_n$ such that 
\[\delta(a)= C_n [V^n\mathbb L, a] + \tilde \delta (a)\] 
for every $a\in \mathcal B(N)$, where if $N$ is finite and $N \mid n$, then $\tilde \delta$ is an inner derivation, and if $N$ is infinite and $n=0$ then $\tilde \delta$ is an approximately inner derivation.
\end{theo}

\section{General unbounded derivations}

This chapter contains our main results: the classification of derivations in $A(N)$ and $B(N)$. The structure of such derivations differs depending on whether $N$ is finite or infinite, and is interesting even for the simplest case of $N=1$, when $A(1)$ is the Toeplitz algebra. The main technique is the use of Fourier series with respect to the $S^1$ action $\rho^\mathbb{K}_\theta$ on $A(N)$, and $\rho^{\mathbb L}_\theta$ on $B(N)$. The Fourier coefficients of derivations are defined in the following way.

 \begin{defin}\label{Fou_comp}
If $d$ is a derivation in $A(N)$, the {\it $n$-th Fourier component} of $d$ is defined as: 
$$d_n(a)= \frac 1{2\pi} \int_0^{2\pi} e^{in\theta} (\rho^{\mathbb K}_{\theta})^{-1}d\rho^{\mathbb K}_{\theta}(a)\; d\theta.$$
 \end{defin}

\begin{defin}\label{Fou_comp1}
If $\delta$ is a derivation in $B(N)$, the {\it $n$-th Fourier component} of $\delta$ is defined as: 
$$\delta_n(b)= \frac 1{2\pi} \int_0^{2\pi} e^{in\theta} (\rho^{\mathbb L}_{\theta})^{-1}\delta\rho^{\mathbb L}_{\theta}(b)\; d\theta.$$
 \end{defin}

We have the following simple observation.

\begin{prop}
If $d$ is a derivation in $A(N)$, then $d_n$ is an $n$-covariant derivation and well-defined on $\mathcal A(N)$.
\end{prop}

\begin{proof}
It is straightforward to see that $d_n$ is a derivation and is well-defined on $\mathcal A(N)$.
The following computation verifies that $d_n$ is $n$-covariant:
\[\begin{aligned} (\rho^{\mathbb K}_{\theta})^{-1}d_n\rho^{\mathbb K}_{\theta}(a)&=\frac 1{2\pi} \int_0^{2\pi} e^{in\phi}  (\rho^{\mathbb K}_{\theta})^{-1}\rho_{\phi}^{-1}d \rho_{\phi}\rho^{\mathbb K}_{\theta}(a)\; d\phi\\
&= \frac 1{2\pi} \int_0^{2\pi} e^{in\phi}  \rho_{\theta + \phi}^{-1}d \rho_{\theta + \phi}(a)\; d\phi .
\end{aligned}\]
Changing to new variable $\theta + \phi$, and using the translation invariance of the measure, it now follows that $(\rho^{\mathbb K}_{\theta})^{-1}d_n\rho^{\mathbb K}_{\theta}(a)= e^{-in\theta} d_n(a)$.
\end{proof}

We have the following key Ces\`aro mean convergence result for Fourier components of $d$, which is more generally valid for unbounded derivations in any Banach algebra with the continuous circle action preserving the domain of the derivation.
\begin{lem} If $d$  is a derivation in $A(N)$ then:
\begin{equation}\label{Ces_eq}
d(a)=\lim_{M \rightarrow \infty } \frac 1{M+1} \sum_{j=0}^M \left(\sum_{n=-j}^j d_n(a)\right),
\end{equation}
for every $a\in \mathcal A(N)$.
\end{lem}
\begin{proof} 
We need to show that:
\begin{equation*}
\frac 1{M+1} \sum_{j=0}^M \left(\sum_{n=-j}^j d_n(a)-d(a)\right) \xrightarrow {M \rightarrow \infty }0 
\end{equation*}
for all $a\in  \mathcal A(N)$. Using the standard Fourier analysis \cite{K} we can write:
\[\frac 1{M+1} \sum_{j=0}^M \left(\sum_{n=-j}^j d_n(a)-d(a)\right)= \frac 1{2\pi} \int_0^{2\pi} F_M(\theta) \left((\rho^{\mathbb K}_{\theta})^{-1}d_n\rho^{\mathbb K}_{\theta}(a)- d(a)\right) d \theta,\]
where: 
$$F_M(\theta)= \frac 1{M+1} \left(\frac{\sin\left(\frac{M+1}2\right)\theta}{\sin\left(\frac{\theta}2\right)}\right)^2$$ 
is the Fej\'er kernel, which is manifestly positive and satisfies: 
$$\frac 1{2 \pi}\int_0^{2\pi} F_M(\theta) d\theta =1.$$ 
Since $(\rho^{\mathbb K}_{\theta})^{-1}d_n\rho^{\mathbb K}_{\theta}(a)- d(a)$ is continuous in $\theta$, given $\epsilon >0$ we can find small $\omega >0$ so that we have estimates:
\[\begin{aligned} \frac 1{2\pi} \int_0^{\omega} F_M(\theta) \|(\rho^{\mathbb K}_{\theta})^{-1}d_n\rho^{\mathbb K}_{\theta}(a)- d(a)\| d \theta &\leq \frac{\epsilon}{3} \;\;\textnormal{ and}\\
\frac 1{2\pi} \int_{2\pi- \omega}^{2\pi} F_M(\theta) \|(\rho^{\mathbb K}_{\theta})^{-1}d_n\rho^{\mathbb K}_{\theta}(a)- d(a)\| d \theta &\leq \frac{\epsilon}{3} .
\end{aligned}\]
Moreover, on the remaining interval we can estimate as follows:
\[\frac 1{2\pi} \int_{\omega}^{2\pi-\omega} F_M(\theta) \|(\rho^{\mathbb K}_{\theta})^{-1}d_n\rho^{\mathbb K}_{\theta}(a)- d(a)\| d \theta \leq \frac {\textrm{const}}{(M+1) \sin^2(\omega /2)}\] 
for some constant in the numerator. Consequently, we can choose $M$ large enough so that we get:
\[\left\|\frac 1{M+1} \sum_{j=0}^M \left(\sum_{n=-j}^j d_n(a)-d(a)\right)\right\| \leq \epsilon,\]
which completes the proof of \eqref{Ces_eq}.
\end{proof}

The first case we consider is a description of derivations for infinite $N$.

\begin{theo}\label{inf_N_class}
Suppose $d$  is a derivation in $A(N)$ with $N$ infinite. Then there exists a unique constant $C$ such that 
\[d(a)= C[\mathbb K, a] + \tilde d(a)\]
where $\tilde d$ is approximately inner.
\end{theo}
\begin{proof} 
Let $d_0$ be the $0$-th Fourier component of $d$. It is an invariant derivation, so by Theorem \ref{ncov_N|n} we have the unique decomposition:
\begin{equation*}
d_0(a)= C d_{0,\mathbb K}(a) + \tilde d_0 (a)= C[\mathbb K, a] + \tilde d_0(a),
\end{equation*}
for every $a\in \mathcal A(N)$, where $\tilde d_0$ is an approximately inner derivation.
From Theorem \ref{n_cov_inner} we have that the Fourier components $d_n$, $n \neq 0$ are inner derivations.   It follows from \eqref{Ces_eq}, by extracting $d_0$, that we have:
\begin{equation*}
d(a)=d_0(a)+\lim_{M\to\infty}\frac 1{M+1} \sum_{j=1}^M \left(\sum_{|n|\leq j,\, n\ne 0}d_n(a)\right).
\end{equation*}
The terms under the limit sign are all finite linear combinations of $n$-covariant derivations and so they are inner derivations themselves, meaning that the limit is approximately inner, which ends the proof.
\end{proof}

In exactly the same way we obtain the corresponding classification result for unbounded derivations in Bunce-Deddens algebras for infinite $N$.
\begin{theo}
Suppose $\delta$  is a derivation in $B(N)$ with $N$ infinite. Then there exists a unique constant $C$ such that 
\[\delta(b)= C[\mathbb L, b] + \tilde \delta(b)\]
where $\tilde \delta$ is approximately inner.
\end{theo}

We now turn to the classification of derivations in Bunce-Deddens and Bunce-Deddens-Toeplitz algebras for finite $N$. We start with the following simple observation.

\begin{lem}\label{center}
If $\delta:  \mathcal B(N) \rightarrow B(N)$ is a derivation, then $\delta(V^N) \in \textnormal{C}^*(V^N)$. 
\end{lem}

\begin{proof}
Applying $\delta$ to the relation $V^N P_{sr}= P_{sr}V^N$, we see that $\delta(V^N)$ commutes with $P_{sr}$ for every $r,s$ and so it must be in C$^*(V^N)$.
\end{proof}

By Propositions \ref{iden_2} and \ref{iden_3}, we know that for finite $N$ we have an isomorphism of C$^*$-algebras: $B(N) \cong C(S^1) \otimes  M_N(\mathbb C)$ and $\mathcal B(N)$ can be identified with  the set of $N$ by $N$ matrix-valued trigonometric polynomials $F(t)$ on $S^1$.  For any $f\in C(S^1)$ we define the following special derivation $\delta_f$ in $B(N)$:
\[\delta_f(F(t))= f(t) \frac 1i \frac d{dt} F(t).\]
Derivations $\delta_f$ are used in the following theorem which gives very concrete and explicit classification of derivations in $B(N)$.
\begin{theo}
Suppose $N$ is finite and $\delta$  is a derivation in $B(N)$. Then there exists a unique $f\in C(S^1)$ such that 
\[\delta = \delta_f + \tilde \delta\]
where $ \tilde \delta$ is inner.
\end{theo}

\begin{proof}
If $p(t)$ is a trigonometric polynomial and $A \in M_N(\mathbb C)$, then we have:
\[\delta(p(t)A) = \delta(p(t))A + p(t) \delta(A).\]
Here $C(S^1) \cong $ C$^*(V^N)$ and by Lemma \ref{center}, there is $f\in C(S^1)$ such that  $\delta(e^{it})= f(t) e^{it}$ and hence we have:
\[\delta(p(t))= f(t) \frac 1i \frac d{dt} p(t).\]
Moreover, given any derivation $\delta: M_N(\mathbb C) \rightarrow C(S^1,M_N(\mathbb C) )$, the following continuous  matrix-valued function $H(t) \in C(S^1,M_N(\mathbb C) )$ given by:
\[H(t)= \frac 1N \sum_{r,s=1}^N \delta(P_{rs})(t) P_{rs}\]
satisfies the easily verifiable relation:
$$\delta(A)(t)=[H(t), A].$$  
Consequently, we have:
\[\delta(p(t))= \left(f(t) \frac 1i \frac d{dt} p(t)\right)A + p(t)[H(t), A] =
\delta_f (p(t)A) + [H(t), p(t)A],\]
which completes the proof.
\end{proof}

It remains to classify derivations in $A(N)$ for finite $N$. For any $f\in C(S^1)$ we define a special derivation $d_f$ in $A(N)$ to be the unique derivation such that:
\begin{equation}\label{d_f_formulas}
d_f(a_{\textnormal{per}}(\mathbb K))=0, \;\; d_f(U)=\frac 1N UT(f(V^N)), \;\; d_f(U^*)=-\frac 1N T(f(V^N))U^*, 
\end{equation}
where $a_{\textnormal{per}}(\mathbb K)$ is any element of $A_{\textrm{diag, per}}(N)$.
Here $T(f(V^N))$ is the Toeplitz operator of formula \eqref{Toep_def}, where $f(V^N)$ is an operator in $\ell^2(\mathbb Z)$ defined by the functional calculus. 
The derivation $d_f$ is given on generators of $\mathcal A(N)$, hence, if it exists it is unique; to see that it is unambiguously defined on all of $\mathcal A(N)$ we need an additional argument.

Let $f(t)=\sum_{n\in\mathbb Z}f_ne^{int}$ be a trigonometric polynomial which we decompose as: 
$$f(t)=f^+(t)+f^-(t),$$
where $f^+(t)=\sum_{n\geq 0}f_ne^{int}$ and $f^-(t)$ has a similar formula, then we claim that we have the following formula for $d_f$:
\begin{equation}\label{df_formula}
d_f(a) = \frac{1}{N}\left[T(f^+(V^N))(\K+I) + (\K+I)T(f^-(V^N)), a\right].
\end{equation}
  To verify \eqref {d_f_formulas} we calculate using Lemma \ref{Toep_lemma}:
\begin{equation*}
Nd_f(a_{per}(\K)) = T\left(\left[f^+(V^N),a_{per}(\mathbb L)\right]\right)(\K+I) + (\K+I)T\left(\left[f^-(V^N),a_{per}(\mathbb L)\right]\right).
\end{equation*}
Since $a_{per}(\mathbb L)$ is $N$ periodic, it commutes with $V^N$, thus the above commutators are zero and hence $d_f(a_{per}(\K)=0$.

Next, notice that $UT(f^+(V^N)) = T(f^+(V^N))U$ since $f^+(V^N)$ only contains nonnegative powers of $V$.  Using this fact, the commutation relation \eqref{the_com_rel},
and Lemma \ref{Toep_lemma} we have
\begin{equation*}
\begin{aligned}
Nd_f(U) 
&=T(f^+(V^N))\left[(\K+I)U - U(\K+I)\right] + \left[(\K+I) - \K UU^*\right]T(f^-(V^N)V) \\
&=\left(T(f^+(V^N)) + T(f^-(V^N))\right)U = T(f(V^N))U.
\end{aligned}
\end{equation*}

For similar reasons as above we have $U^*T(f^-(V^N)) = T(f^-(V^N))U^*$.  Using this, the commutation relation \eqref{the_com_rel}, and again Lemma \ref{Toep_lemma}, we obtain the last part of formula \eqref{d_f_formulas}.
Thus this completes the proof of existence of $d_f$ for polynomial $f$.
It is clear from those formulas that $d_f$ is a well-defined derivation $\mathcal A(N) \rightarrow A(N)$. 

For a general $f\in C(S^1)$ we use an approximation argument to construct $d_f$. Namely if $\{f^M\}$ is a sequence of trigonometric polynomials converging uniformly to $f$ then, by formulas \eqref{d_f_formulas}, the sequence of derivations $\{d_{f^M}\}$ converges on generators of $\mathcal A(N)$, and hence it converges for every $a\in\mathcal A(N)$. The limit, which must be a derivation in $A(N)$, gives a construction of $d_f$.
Derivations $d_f$ are used in the theorem below.

Compared to the proof of Theorem \ref{inf_N_class}, the classification of derivations in $A(N)$ for finite $N$ gets more complicated since in this case a derivation may have infinitely many non-inner Fourier components. To handle those difficulties we need the following lemma which is more generally valid for unbounded derivations in any algebra if the domain is finitely generated.

\begin{lem}\label{app_der_lem}
If $N$ is finite, $d$  is a derivation in $A(N)$, and there is a sequence $\{d^M\}$ of approximately inner derivations such that for every $a\in \mathcal A(N)$:
$$d(a)=\lim_{M\to\infty}d^M(a),
$$
then $d$ is also approximately inner.
\end{lem}

\begin{proof} For finite $N$ the algebra $\mathcal A(N)$ is finitely generated; for example we can choose the following set of generators:
\begin{equation*}
G:=\{U, U^*, e_N(\K)\},
\end{equation*}
where the sequence $e_N(k)$ was defined in \eqref{eN_def}. Also, $d^M$ are approximately inner which means that there is a sequence $\{z^{M,W}\}$ of elements of $A(N)$ such that for every $a\in \mathcal A(N)$:
\begin{equation*}
d^M(a)=\lim_{W\to\infty}[z^{M,W},a].
\end{equation*}
For every positive integer $j$ we can choose $M_j$ such that for every $a\in G$ we have:
\begin{equation*}
\left\|d(a)-d^{M_j}(a)\right\|\leq \frac{1}{2j},
\end{equation*}
which can be done because the generating set $G$ is finite. Then choose $W_j$ such that for every $a\in G$ we have:
\begin{equation*}
\left\|d^{M_j}(a)-[z^{M_j,W_j},a]\right\|\leq \frac{1}{2j}.
\end{equation*}
By the triangle inequality we obtain:
\begin{equation*}
\left\|d(a)-[z^{M_j,W_j},a]\right\|\leq \frac{1}{j},
\end{equation*}
which means that we have:
\begin{equation*}
d(a)=\lim_{j\to\infty}[z^{M_j,W_j},a]
\end{equation*}
for every $a\in G$, which by the Leibniz identity implies the above convergence for every $a\in \mathcal A(N)$. Consequently, $d$ is approximately inner, finishing the proof.
\end{proof}

With this preparation we are now ready to state our classification result for derivations in $A(N)$ with finite $N$.
\begin{theo}
Suppose $N$ is finite and $d$  is a derivation in $A(N)$. Then there exists unique function $f\in C(S^1)$ such that: 
\[d = d_f + \tilde d,\]
where $ \tilde d$ is approximately inner and $d_f$ is defined by formula \eqref{d_f_formulas}.
\end{theo}

\begin{proof}
Consider the derivation $[d]:  \mathcal B(N) \rightarrow B(N)$ given by: 
$$[d](a+\mathcal K)= d(a)+\mathcal K.$$ 
It is easy to see from Definitions \ref{Fou_comp} and \ref{Fou_comp1} that we have the following equality for the Fourier components:
$$[d]_n= [d_n].$$

To construct the function $f$ in the statement of the theorem we notice that Lemma \ref{center} states that $[d](V^N)\in \textnormal{C}^*(V^N)$ and hence we can write: 
$$[d](V^N)=f(V^N)V^N$$ 
for some $f\in C(S^1)$. It follows that we have:
\begin{equation}\label{d_n_one}
[d]_n(V^N)=\begin{cases}
f_jV^{jN+N} &  \textnormal{ if } n=jN\\
0 &  \textnormal{otherwise,}
\end{cases}
\end{equation}
where $f_j$ are the Fourier coefficients of $f$. 

On the other hand, from Theorem \ref{n_cov_der_formula} we know that 
\[d_n(U^N)=[U^n \beta_n(\mathbb K), U^N]= U^{n+N}\left(\alpha_n(\mathbb K+(N-1)I) + \cdots + \alpha_n(\mathbb K)\right). \]

Next, for $N\mid b$, we decompose $\alpha_n(k)$ as in the proof of Theorem \ref{ncov_N|n}: 
\[\alpha_n(k)= \alpha_{n,0}(k)+C_n + \alpha_{n, \textnormal{per}}(k),\]
where $C_n$ is a constant, $\alpha_{n,0}(k) \in c_0$, $  \alpha_{n, \textnormal{per}}(k+N)=\alpha_{n, \textnormal{per}}(k)$ and 
$$\sum_{k=0}^{N-1}  \alpha_{n, \textnormal{per}}(k) =0.$$
This decomposition is also valid for $N\nmid n$ but with $C_n=0$ by Theorem \ref{n_cov_inner}.
It follows that we have: 
\[d_n(U^N)=U^{n+N}\left(\alpha_{n,0} (\mathbb K +(N-1)I)+ \cdots + \alpha_{n,0}(\mathbb K) + NC_n\right),\]
and consequently, we obtain:
\begin{equation}\label{d_n_two}
[d_n](V^N)= NC_nV^{n+N} = \begin{cases}
NC_{jN}V^{jN+N}  &  \textnormal{ if } n=jN\\
0 &  \textnormal{otherwise.}
\end{cases}
\end{equation}
Comparing equation \eqref{d_n_one} and \eqref{d_n_one} implies the following formulas for constants $C_n$:
\begin{equation*}
C_n=\begin{cases}
\frac{1}{N}f_j &  \textnormal{ if } n=jN\\
0 &  \textnormal{otherwise.}
\end{cases}
\end{equation*}
It then follows from the formula \eqref{df_formula} that we have:
\begin{equation}\label{df_fourier}
(d_f)_n(a)= \frac{1}{N}f_j d_{n,\mathbb K}(a)=\begin{cases}
[U^n C_n (\mathbb K +I), a] &  \textnormal{ if } n \geq 0\\
[C_n (\mathbb K +I) (U^*)^{-n}, a] &  \textnormal{ if } n <0.
\end{cases}
\end{equation}

As in the proof of Theorem \ref{ncov_N|n} we decompose $\beta_n$ using:
\[\beta_{n,0}(k):= \sum_{j=0}^{k}\alpha_{n,0}(j), \;\;\; \beta_{n, \textnormal{per}}(k):=  \sum_{j=0}^{k}\alpha_{n,\textnormal{per}}(j).\]
This gives the following formulas for the Fourier components of the difference between $d$ and $d_f$:
\[\tilde d_n(a):=(d-d_f)_n(a)= \begin{cases}
[U^n \left(\beta_{n,0}(\mathbb K) + \beta_{n, \textnormal{per}}(\mathbb K)\right), a] &  \textnormal{ if } n \geq 0\\
[\left(\beta_{n,0}(\mathbb K) + \beta_{n, \textnormal{per}}(\mathbb K)\right) (U^*)^{-n}, a] &  \textnormal{ if } n <0.
\end{cases}\]
From Theorem \ref{n_cov_inner} we know that $\tilde d_n$ is an inner derivation if $N\nmid n$. If we denote by $(\tilde d_0)_n$ and  $(\tilde d_{\textnormal{per}})_n$ the following derivations on $ \mathcal A(N)$:
\[(\tilde d_0)_n(a)=  \begin{cases}
[U^n\beta_{n,0}(\mathbb K) , a] &  \textnormal{ if } n \geq 0\\
[\beta_{n,0}(\mathbb K)  (U^*)^{-n}, a] &  \textnormal{ if } n <0,
\end{cases}\] 
\[(\tilde d_{\textnormal{per}})_n(a)= \begin{cases}
[U^n \beta_{n, \textnormal{per}}(\mathbb K) , a] &  \textnormal{ if } n \geq 0\\
[\beta_{n, \textnormal{per}}(\mathbb K) (U^*)^{-n}, a] &  \textnormal{ if } n <0
\end{cases}
\]
then, when $N\mid n$, we know that $(\tilde d_{\textnormal{per}})_n$ is inner while  $(\tilde d_0)_n$ is approximately inner from Theorem \ref{ncov_N|n}. 
To conclude that $\tilde d$ is approximately inner we first use formula \eqref{Ces_eq} which says, in view of the above discussion, that $\tilde d$ is a limit of approximately inner derivations. Consequently, using Lemma \ref{app_der_lem}, we see that $\tilde d$ is approximately inner.

To show the uniqueness of this decomposition, it is sufficient to prove that $d_f$ is approximately inner if and only if $f=0$. If $f=0$, it is clear that $d_f$ is approximately inner.  To prove the converse statement, notice that if $d_f$ is approximately inner then so are the Fourier components $(d_f)_n$, which by formula \eqref{df_fourier} are proportional to derivations  $d_{n,\mathbb K}$, which in turn were proved in Theorem \ref{ncov_N|n} not to be approximately inner. This gives a contradiction and finishes the proof of the theorem.
\end{proof}

\section{Implementations}

The purpose of this section is to investigate implementations of unbounded derivations in Bunce-Deddens algebras $B(N)$ as operators in Hilbert spaces. 
This study is inspired by the following noncommutative geometry concept of first order elliptic operator with respect to a C$^*$-algebra: an unbounded operator $D$ acting in a Hilbert space $\mathcal H$ which carries a representation $\pi$ of a C$^*$-algebra $A$ is called a first order elliptic operator with respect to $A$ if it satisfies two properties:
\begin{enumerate}
\item $[D,\pi(a)]$ is bounded for all $a$ in some dense $^*$-subalgebra $\mathcal A$ of $A$.
\item $D$ has a compact parametrices, which by the appendix of \cite{KMR1} is equivalent to the two operators $(I + D^*D)^{-1/2}$ and $(I + DD^*)^{-1/2}$ being compact operators.
\end{enumerate}
Such first a order elliptic operator with respect to $A$ is a key component of the notion of a spectral triple in noncommutative geometry; see \cite{CPR}.
 
If a first order elliptic operator $D$ is an implementation of a densely-defined unbounded derivation in $A$ then the first condition of the above definition is automatically satisfied. Hence we are mainly interested in establishing when implementations of derivations in $B(N)$ have compact parametrices. We only consider here the representations of $B(N)$ in Hilbert spaces obtained from the GNS construction as those are the most geometrical representations of those algebras.

In \cite{KMR1} and \cite{KMR2} implementations of 0-covariant, that is invariant, and 1-covariant derivations in the quantum disk (Toeplitz algebra) and the quantum annulus were studied to see if it was possible to construct spectral triples on those quantum domains.  Here we continue this analysis for $n$-covariant derivations in $B(N)$.

A state $\tau: B(N) \to \mathbb C$ is called a $\rho^{\mathbb L}_{\theta}$-{\it invariant state} on $B(N)$ if for all $a\in B(N)$ it satisfies the following:
\[\tau(\rho^{\mathbb L}_{\theta}(a))= \tau(a).\]  
It is not difficult to describe the $\rho^{\mathbb L}_{\theta}$-invariant states on $B(N)$. 
To do this we use the identification $B(N) \cong C(\mathbb Z/N\mathbb Z) \rtimes_{\sigma} \mathbb Z$; see Proposition \ref{cross_iden}.
There is a natural expectation $E: B(N) \rightarrow C(\mathbb Z/N\mathbb Z)$, a positive, linear map such that $E^2=E$. For an element 
\[b= \sum\limits_{n \in \mathbb Z} V^n b_n(x) \in \mathcal B(N),\]  
see \eqref{pol_In_B(N)}, it is given by: 
$$E(b)= \frac{1}{2\pi}\int_0^{2\pi}\rho^{\mathbb L}_{\theta}(b)\,d\theta= b_0(x) \in  C(\mathbb Z/N\mathbb Z).$$
Since $C(\mathbb Z/N\mathbb Z)$ is the fixed point algebra for $\rho^{\mathbb L}_{\theta}$, we immediately obtain the following observation:
if $\tau : B(N)\to\mathbb C$ is a $\rho^{\mathbb L}_{\theta}-$invariant state on $B(N)$ then there exists a state $t : C(\mathbb Z/N\mathbb Z)\to\mathbb C$ such that: 
$$\tau(b) = t(E(b)).$$  
Conversely given a state $t : C(\mathbb Z/N\mathbb Z)\to\mathbb C$, then $\tau(b) = t(E(b))$ defines a $\rho_\theta-$invariant state on $B(N)$.
Therefore the invariant states are given by probabilistic measures on $\mathbb Z/N\mathbb Z$.

We will concentrate below on the following two most interesting and natural $\rho^{\mathbb L}_{\theta}$-invariant states on $B(N)$, namely $\tau_0$ and $\tau_{\textnormal{Haar}}$ defined by: 
\[\tau_0(b)= E(b)(0)\quad\textrm{ and }\quad \tau_{\textnormal{Haar}}(b)= \int_{\mathbb Z/N\mathbb Z}E(b)(x)\ d_Hx, \]
where $d_Hx$ is the unique normalized Haar measure. Denote by $\mathcal{H}_0$ and  $\mathcal{H}_{\textnormal{Haar}}$ the GNS Hilbert spaces corresponding to $\tau_0$ and $\tau_{\textnormal{Haar}}$ respectively and let $\pi_0, \pi_{\textnormal{Haar}}$ be the corresponding representations.
We have the following concrete description of those Hilbert spaces and representations.

\begin{prop}\label{gns_hilbert_spaces}
The GNS Hilbert spaces $\mathcal{H}_0$ and  $\mathcal{H}_{\textnormal{Haar}}$ are naturally isomorphic to the following:
\begin{equation*}
\mathcal{H}_0 \cong \ell^2(\mathbb Z)\quad\textrm{ and }\quad \mathcal{H}_{\textrm{Haar}} \cong L^2(\mathbb Z\times\mathbb{Z}/N\mathbb{Z}).
\end{equation*}

The representation $\pi_0: B(N) \rightarrow \mathcal B(\ell^2(\mathbb Z))$ is the defining representation of $B(N)$, i.e. $\pi_0(a)=a$ for all $a\in B(N)$.

The representation $\pi_{\textnormal{Haar}}: B(N) \rightarrow \mathcal B(L^2(\mathbb Z\times\mathbb{Z}/N\mathbb{Z}))$ is completely described by:
\begin{equation*}
\begin{aligned}
&1.\ \pi_{\textrm{Haar}}(V)f(m,x)= f(m-1,x) \\
&2.\ \pi_{\textrm{Haar}}(a(q(\mathbb L)))f(m,x)=a(x+m)f(m,x),
\end{aligned}
\end{equation*} 
where $f(m,x)\in L^2(\mathbb Z\times\mathbb{Z}/N\mathbb{Z})$ and $a(x)\in C(\mathbb Z/N\mathbb Z)$.
\end{prop}

\begin{proof}
To properly identify the Hilbert space 
\[
\mathcal{H}_0= \overline{B(N)/\{b\in B(N): \tau_0(b^*b)=0\}} 
\]
we must study $\tau_0(b^*b)=0$ for $b\in B(N)$.  In fact, due to the continuity of $\tau_0$, we only need to work on the dense subalgebra $\mathcal{B}(N)$.  For any $b\in\mathcal{B}(N)$ given by equation \eqref{pol_In_B(N)}, a straightforward calculation yields:
\begin{equation*}
\tau_0(b^*b) = \sum_{n\in\mathbb Z}|b_{n}(0)|^2.
\end{equation*}
Therefore, if $\tau_0(b^*b) = 0$, it follows that $b_{n}(0)=0$ for all $n$.   Then the formula:
$$
\mathcal{H}_0\ni[b]\mapsto \{b_n(0)\}\in\ell^2(\mathbb{Z})
$$
gives an isomorphism $\mathcal{H}_0 \cong \ell^2(\mathbb{Z})$, similar to the proof of Proposition $5.4$ in \cite{KMR1}. Notice that the class $[I]$ in the completion of the quotient $\overline{B(N)/\{b\in B(N): \tau_0(b^*b)=0\}}$ corresponds to the basis element $E_0$ in $ \ell^2(\mathbb{Z})$.
From the formula:
\begin{equation*}
Vb = \sum_{n\in\mathbb Z}V^nb_{n-1}(\mathbb{L}),
\end{equation*}
it follows that we have:
$$\pi_0(V)[b]= \{b_{n-1}(0)\}_{n\in\mathbb{Z}}.$$ 
An analogous calculation shows: 
$$\pi_0(a(\mathbb L))[b]=\{a(n)b_{n}(0)\}_{n\in\mathbb{Z}}$$ 
for $a(\mathbb L)\in B_{\textnormal{diag}}(N)$.  This proves the first part of the proposition. 

In the second example we have $\tau_{\textrm{Haar}}(b^*b)=0$ if and only if $b=0$. If $b\in\mathcal B$ is given by:
$$b=\sum_{n \in \mathbb Z} V^n b_{n} (\mathbb L)=
\sum_{n \in \mathbb Z} V^n f_{n} (q(\mathbb L))
$$
then the corresponding function in $L^2(\mathbb Z\times\mathbb{Z}/N\mathbb{Z})$ is given by:
$$[b](m,x)=f_m(x).
$$
Otherwise calculations with
$\tau_{\textrm{Haar}}$ are very similar.
\end{proof}

We remark here briefly that because $B(N)$ is defined as the quotient of $A(N)$ with the ideal of compact operators, an invariant state on $B(N)$ lifts to an invariant state on $A(N)$. The corresponding GNS Hilbert spaces are the same as for  $B(N)$, with compact operators represented trivially.

In general, for a GNS Hilbert space $\mathcal{H}_\tau$  of $B(N)$ with respect to a state $\tau$ we have that $B(N)\subseteq \mathcal{H}_\tau$ is dense in $\mathcal{H}_\tau$ and $[I]\in \mathcal{H}_\tau$ is cyclic. 
Consequently,  the subspace
$$\mathcal{D}_\tau := \pi_\tau(\mathcal{B}(N))\cdot[I]$$
 is dense in $\mathcal{H}_\tau$.  Define $V_{\tau ,\theta} : \mathcal{H}_\tau\to \mathcal{H}_\tau$ via the equation:
 $$V_{\tau ,\theta}[b] = [\rho^\mathbb{L}_\theta(b)].$$  
Notice that for every $\theta$, the operator $V_{\tau ,\theta}$ extends to a unitary operator in $\mathcal{H}_\tau$.   Moreover by a direct calculation we get:
\begin{equation*}
V_{\tau ,\theta}\pi_\tau(b)V_{\tau ,\theta}^{-1} = \pi_\tau(\rho^\mathbb{L}_\theta(b)),
\end{equation*}
meaning that $V_{\tau ,\theta}$ is an implementation of $\rho^\mathbb{L}_\theta$.
It follows from the definitions that we have the following inclusions:
$$V_{\tau ,\theta}(\mathcal{D}_\tau)\subseteq\mathcal{D}_\tau \textrm{ and }\pi_\tau(\mathcal{B}(N))(\mathcal{D}_\tau)\subseteq\mathcal{D}_\tau.$$

Let $\delta$ be an $n$-covariant derivation in $B(N)$ and let $\tau$ be a $\rho^{\mathbb{L}}_\theta-$invariant state. Implementations of $\delta$ in the GNS Hilbert space $\mathcal H_{\tau}$ are defined in the following way.
\begin{defin}
An operator $D_\tau :\mathcal{D}_\tau \to \mathcal{H}_\tau$ is called a {\it covariant implementation} of an $n$-covariant derivation $\delta$ if 
$$[D_\tau, \pi_\tau(b)] = \pi_\tau(\delta(b))$$ and 
$$V_{\tau,\theta} D_\tau V_{\tau,\theta}^{-1} = e^{in\theta} D_\tau.$$
\end{defin}

Below we find all covariant implementations of $n$-covariant derivations on the two GNS Hilbert spaces $\mathcal{H}_0$ and $\mathcal{H}_{\textrm{Haar}}$ of Proposition \ref{gns_hilbert_spaces}, and establish when they have compact parametrices. We start by recapping Theorem \ref{BDncov} with additional details needed for the formulation of the implementation results. 

Any $n$-covariant derivation $\delta$ in $B(N)$ is of the form:
$$\delta(b)=[V^n\eta_n(\mathbb L),b],
$$
where for $N$ infinite, $n\ne 0$ and $N$ finite and $N\nmid n$ the operator $\eta_n(\mathbb L)$ is in $B_{\textnormal{diag}}(N)$; hence it comes from a function $h_n$ in $C(\mathbb Z/N\mathbb Z)$, so that we have:
$$\eta_n(\mathbb L)=h_n(q(\mathbb L)).$$
In other cases the increment 
$$\gamma_n(\mathbb L):=\eta_n(\mathbb L)-\eta_n(\mathbb L-I)$$ 
is in $B_{\textnormal{diag}}(N)$, so it can be written as: 
$$\gamma_n(\mathbb L)=g_{n}(q(\mathbb L))$$ 
for some $g_{n}(x)\in C(\mathbb Z/N\mathbb Z)$. It follows that there is a constant $C_n$ such that we have decompositions:
$$g_n(x)=C_n+\tilde g_n(x) \textrm{ and } \eta_n(\mathbb L)=C_n\mathbb L +\tilde\eta_n(\mathbb L),
$$
where the function $\tilde{g}_{n}(x)\in C(\mathbb Z/N\mathbb Z)$  satisfies the property:
\begin{equation*}
\int_{\mathbb{Z}/N\mathbb{Z}}\tilde{g}_{n}(x)\ d_H x = 0,
\end{equation*}
and we have:
$$
\tilde g_n(q(\mathbb L))=\tilde\eta_n(\mathbb L)-\tilde\eta_n(\mathbb L-I).
$$
When $N$ is infinite and $n=0$, in general, it is possible for $\tilde\eta_0(l)$ to be unbounded. However, when $N$ is finite and $N\mid n$ then $\tilde\eta_n(l)$ must be in the finite dimensional vector space $C(\mathbb Z/N\mathbb Z)$, and so we have:
$$\tilde\eta_n(\mathbb L)=\tilde h_{n}(q(\mathbb L))$$ 
for some $\tilde h_{n}(x)\in C(\mathbb Z/N\mathbb Z)$. All of this notation is used in the following implementation statements.

\begin{theo}
Any covariant implementation $D_{\tau_0} :\mathcal{D}_{\tau_0} \to \ell^2(\mathbb{Z})$ of an $n$-covariant derivation $\delta$ in $B(N)$ is of the form:
\begin{equation*}
D_{\tau_0} =\left\{
\begin{aligned} 
&V^n\eta_n(\mathbb{L}) &&\textrm{for }n\ne0\\
&\eta_0(\mathbb{L}) + c\cdot I &&\textrm{for }n=0,\\
\end{aligned}\right.
\end{equation*}
with arbitrary constant $c$ for $n=0$.
If $N$ is infinite and $n\neq0$ or if $N$ is finite and $N\nmid n$, then the operator $D_{\tau_0}$ is bounded, so it does not have compact parametrices. In all other cases, $\eta_n(l)\to\infty$ as $l\to\infty$ is a necessary and sufficient condition for $D_{\tau_0}$ to have compact parametrices.
\end{theo}

\begin{proof}
It is easy to see that $\mathcal{D}_{\tau_0}$ coincides with $c_{00}\subseteq \ell^2(\mathbb{Z})$. The formulas for $D_{\tau_0}$ follow from simple calculations, just like in \cite{KMR1}. See also the next theorem for more details of similar calculations in the Haar measure state case. 

From the appendix of \cite{KMR1}, $D_{\tau_0}$ has compact parametrices if and only if $(I + D_{\tau_0}^*D_{\tau_0})^{-1/2}$ and $(I + D_{\tau_0}D_{\tau_0}^*)^{-1/2}$ are compact operators.  A direct calculation yields the following formula:
\begin{equation*}
I + D_{\tau_0}^*D_{\tau_0} =\left\{
\begin{aligned}
&I + |\eta_n|^2(\mathbb{L}) &&\textrm{for }n\neq0 \\
&(1+c^2)\cdot I + 2\textrm{Re }\eta_0(\mathbb{L}) + |\eta_0|^2(\mathbb{L}) &&\textrm{for }n=0, \\  
\end{aligned}\right.
\end{equation*}
which is a diagonal operator for all $n$.  Therefore, it follows that $(I + D_{\tau_0}^*D_{\tau_0})^{-1/2}$ is compact if and only if $\eta_n(l)$ goes to infinity, in particular when $C_n\ne 0$. An analogous computation works for $(I + D_{\tau_0}D_{\tau_0}^*)^{-1/2}$, thus completing the proof.
\end{proof}

Similar analysis can also be performed for implementations of $n$-covariant derivations in the GNS Hilbert space corresponding to the invariant state on $B(N)$ determined by the Haar measure on $\mathbb Z/N\mathbb Z$.

\begin{theo}
There exists a function $\psi(x)\in L^2(\mathbb{Z}/N\mathbb{Z},d_Hx)$  such that any implementation $D_{\tau_{\textrm{Haar}}} : \mathcal{D}_{\tau_{\textrm{Haar}}} \to  L^2(\mathbb Z\times\mathbb{Z}/N\mathbb{Z})$ of $\delta$ is of the form: 
\begin{equation*}
\left(D_{\tau_{\textrm{Haar}}}f\right)(m,x) = h_{n}(x+m-n)f(m-n,x) + (\psi(x) - h_{n}(x))f(m-n,x+n), 
\end{equation*}
if $N$ is infinite, $n\neq0$, or if $N$ is finite, $N\nmid n$, and
\begin{equation*}
\left(D_{\tau_{\textrm{Haar}}}f\right)(m,x) = \left(C_0 m+ (\tilde{g}_{0}(x+m-1)+\cdots+\tilde{g}_{0}(x)) + \psi(x)\right)f(m,x), 
\end{equation*}
if $N$ is infinite, $n=0$, or 
\begin{equation*}
\left(D_{\tau_{\textrm{Haar}}}f\right)(m,x) = \left(C_n\cdot(m-n)+\tilde h_{n}(x+m) - \tilde h_{n}(x) +\psi(x) \right)f(m-n,x), 
\end{equation*}
if $N$ is finite, $N\mid n$.

For $N$ finite and $N\mid n$ a necessary and sufficient condition for $D_{\tau_{\textrm{Haar}}}$ to have compact parametrices is $C_n\neq0$.  In all other cases $D_{\tau_{\textrm{Haar}}}$ does not have compact parametrices.
\end{theo}

\begin{proof}
First notice that $[I] = \chi_0(m,x)$ where $\chi_0(m,x) =1 $ when $m=0$ and zero for all other values of $m$.  Given $b\in\mathcal{B}(N)$ we compute as follows:
\begin{equation*}
\begin{aligned}
D_{\tau_{\textrm{Haar}}}[b] &= D_{\tau_{\textrm{Haar}}}\pi_{\textrm{Haar}}(b)[I] = [D_{\tau_{\textrm{Haar}}},\pi_{\textrm{Haar}}(b)][I] + \pi_{\textrm{Haar}}(b)D_{\tau_{\textrm{Haar}}}[I] \\
&=\pi_{\textrm{Haar}}(\delta(b))[I] + \pi_{\textrm{Haar}}(b)D_{\tau_{\textrm{Haar}}}[I].
\end{aligned} 
\end{equation*}
Applying the covariance condition $V_{\tau_{\textrm{Haar}},\theta}D_{\tau_{\textrm{Haar}}}V_{\tau_{\textrm{Haar}},\theta}^{-1} = e^{in\theta}D_{\tau_{\textrm{Haar}}}$ to $[I] = \chi_0(m,x)$ shows that there exists a function $\psi(x)\in L^2(\mathbb{Z}/N\mathbb{Z},d_Hx)$ such that 
$$D_{\tau_{\textrm{Haar}}}\chi_0(m,x) = \psi(x)(\pi_{\textrm{Haar}}(V^n)\chi_0)(m,x).$$  
It follows that we have the formula:
\begin{equation*}
\pi_{\textrm{Haar}}(b)D_{\tau_{\textrm{Haar}}}[I](m,x) = \psi(x)(\pi_{\textrm{Haar}}(bV^n)\chi_0)(m,x)=\psi(x)[b](m-n, x+n),
\end{equation*}
because of the following calculation with Fourier components of $b$:
$$bV^n=\sum_{m \in \mathbb Z} V^{m+n} f_{m} (q(\mathbb L)+n\cdot I)=\sum_{m \in \mathbb Z} V^{m} f_{m-n} (q(\mathbb L)+n\cdot I).
$$
This implies the following general expression of the operator $D_{\tau_{\textrm{Haar}}}$: 
\begin{equation*}
\begin{aligned}
(D_{\tau_{\textrm{Haar}}}[b])(m,x) &=
\left[\sum_{m\in\mathbb Z}V^{m}\eta_{n}(\mathbb{L} + (m-n)\cdot I)b_{m-n}(\mathbb{L})  - \eta_{n}(\mathbb{L})b_{m-n}(\mathbb{L}+n\cdot I)\right](m,x)\\
&+\psi(x)[b](m-n, x+n).
\end{aligned}
\end{equation*}

If $N$ is infinite  and $n\neq 0$ or if $N$ is finite and $N\nmid n$ then $\eta_{n}(\mathbb{L})$ is in $B_{\textrm{diag}}(N)$ hence   it comes from a function $h_n(x)$ in $C(\mathbb Z/N\mathbb Z)$. Consequently, we have the formula:
\begin{equation*}
\left(D_{\tau_{\textrm{Haar}}}[b]\right)(m,x) = h_{n}(x+m-n)[b](m,x) + (\psi(x) - h_{n}(x))[b](m,x+n). 
\end{equation*}
The first and last terms of the above expression (those containing $h_n(x)$) are bounded operators and hence $D_{\tau_{\textrm{Haar}}}$ has compact parametrices if and only if the middle term has compact parametrices by the results in the appendix of \cite{KMR1}. That term is unitarily equivalent to the operator:
$$f(m,x)\mapsto \psi(x)f(m,x),
$$
which for every $m$ is the multiplication operator by an $L^2$-function in $L^2(\mathbb Z/N\mathbb Z,d_Hx)$ and
therefore $D_{\tau_{\textrm{Haar}}}$ can not have compact parametrices. 

In the case when $N$ is infinite and $n=0$, there is in general no function $h_0(x)$ such that $\eta_0(\mathbb L)=h_0(q(\mathbb L))$, and so we write the difference $\eta_0(\mathbb L+m\cdot I) - \eta_0(\mathbb L)$ in terms of $\eta_n(\mathbb L)-\eta_n(\mathbb L-I)=\gamma_n(\mathbb L)=g_0(q(\mathbb L))=C_n\cdot I+\tilde g_0(q(\mathbb L))$ to obtain the following expression:
\begin{equation*}
\left(D_{\tau_{\textrm{Haar}}}[b]\right)(m,x) = \left(C_0 m+ (\tilde{g}_{0}(x+m-1)+\cdots+\tilde{g}_{0}(x)) + \psi(x)\right)[b](m,x). 
\end{equation*}
As in the first case, since for each fixed $m$ the above formula is a diagonal operator that is a multiplication by a $L^2$-function, it is therefore impossible for $D_{\tau_{\textrm{Haar}}}$ to have compact parametrices.  

Finally in the last case, when $N$ is finite and $N\mid n$, the Hilbert space $L^2(\mathbb Z/N\mathbb Z,d_Hx)$ is now a finite dimensional Hilbert space.   Hence, we can decompose $\eta_n(\mathbb L)$ as follows: 
$$\eta_n(\mathbb L)=C_n\mathbb L+ \tilde g_n(q(\mathbb L)).$$ 
Using $N$-periodicity in $x$ we arrive at the expression:
\begin{equation*}
\left(D_{\tau_{\textrm{Haar}}}[b]\right)(m,x) = \left(C_n\cdot(m-n)+\tilde h_{n}(x+m) - \tilde h_{n}(x) +\psi(x) \right)[b](m-n,x).
\end{equation*}
Notice that since $\tilde h_{n}(x)$ is $N$-periodic then $\tilde h_{n}(x+m)$ is uniformly bounded in $m$ and $x$.
It now follows that $D_{\tau_{\textrm{Haar}}}$ have compact parametrices if and only if $C_n\neq0$.  This completes the proof.
\end{proof}

\end{document}